\documentclass[10pt]{amsart}

\usepackage{graphicx}
\usepackage{amsmath}
\usepackage{amssymb}
\usepackage{dsfont}
\usepackage{enumitem}

\numberwithin{equation}{section}
\def\RR{{\mathbb R}}
\def\MM{{\mathbb M}^{m\times N}}

\def\spt{{\rm supp}}

\def\wto{\rightharpoonup}
\def\dto{{\scriptstyle\buildrel{\scriptstyle\longrightarrow}
\over{{}_{\scriptstyle\longrightarrow}}}}
\def\eps{\varepsilon}

\def\cupp{\mathop{\cup}}

\def\ell{l}

\def\mint{{{\bf-}\!\!\!\!\!\!\hspace{-.1em}\int}}

\def\limsup{\mathop{\overline{\lim}}}
\def\liminf{\mathop{\underline{\lim}}}

\def\endproof{$\blacksquare$}

\newtheorem{theorem}{Theorem}[section]
\newtheorem{lemma}[theorem]{Lemma}
\newtheorem{proposition}[theorem]{Proposition}
\newtheorem{corollary}[theorem]{Corollary}

\theoremstyle{definition}
\newtheorem{definition}[theorem]{Definition}

\theoremstyle{remark}
\newtheorem{remark}[theorem]{Remark}

\title[Relaxation of variational integrals in metric Sobolev spaces]{On the relaxation of variational integrals in metric Sobolev spaces}

\author{\sc Omar Anza Hafsa}
\address{UNIVERSITE DE NIMES, Laboratoire MIPA, Site des Carmes, Place Gabriel P\'eri, 30021 N\^\i mes, France. \newline LMGC, UMR-CNRS 5508, Place Eug\`ene Bataillon, 34095 Montpellier, France.}
\email{omar.anza-hafsa@unimes.fr}

\author{Jean-Philippe Mandallena}
\address{UNIVERSITE DE NIMES, Laboratoire MIPA, Site des Carmes, Place Gabriel P\'eri, 30021 N\^\i mes, France. \newline LMGC, UMR-CNRS 5508, Place Eug\`ene Bataillon, 34095 Montpellier, France.}
\email{jean-philippe.mandallena@unimes.fr}


\begin{document}

\begin{abstract}
We give an extension of the theory of relaxation of variational integrals in classical Sobolev spaces to the setting of metric Sobolev spaces. More precisely, we establish a general framework to deal with the problem of finding an integral representation for  ``relaxed" variational functionals of variational integrals of the calculus of variations in the setting of metric measure spaces. We prove integral representation theorems, both in the convex and non-convex case, which extend and complete previous results in the setting of euclidean measure spaces to the  setting of metric measure spaces. We also show that these integral representation theorems can be applied in the setting of Cheeger-Keith's differentiable structure.

\end{abstract}

\maketitle


\section{Introduction}

Let $(X,d,\mu)$ be a metric measure space, where $(X,d)$ is a separable and compact metric space and $\mu$ is a positive Radon measure on $X$. Let $p\in]1,\infty[$ be a real number and let $\{L_x\}$ be a field of Carath\'eodory integrands over $X$ (see the begining of \S 2.2 for more details) assumed to be both $p$-coercive, see \eqref{p-coercivity}, and of $p$-polynomial growth, see \eqref{p-polynomial-growth}. Let $m\geq 1$ be an integer and let $\mathcal{O}(X)$ be the class of all open subsets of $X$. In this paper, we are concerned with the problem of finding an integral representation for the ``relaxed" variational functional $\overline{E}:W^{1,p}_\mu(X;\RR^m)\times\mathcal{O}(X)\to[0,\infty]$ given by
$$
\overline{E}(u;A):=\inf\left\{\liminf_{n\to\infty}\int_A L_x(\nabla u_n(x))d\mu(x):\mathcal{A}(X;\RR^m)\ni u_n\to u\hbox{ in }L^p_\mu(X;\RR^m)\right\},
$$
where $\mathcal{A}(X;\RR^m):=[\mathcal{A}(X)]^m$, with $\mathcal{A}(X)$ a subalgebra of the algebra of all continuous functions from $X$ to $\RR$, which contains the constants and enough cut-off functions (see the begining of \S 2.1 for more details), and the operator $\nabla$, from $\mathcal{A}(X;\RR^m)$ to $L^\infty_\mu(X;\RR^N)$ with $N\geq 1$ an integer, is a gradient over $\mathcal{A}(X;\RR^m)$, see \eqref{Gradient-Property}. For example, $\mathcal{A}(X)$ can be the algebra of all restrictions to the closure of a bounded open subset of $\RR^N$ of $C^1$-functions from $\RR^N$ to $\RR$ or, more generally, the algebra of all Lipschitz functions from $X$ to $\RR$ (see Remark \ref{ReMaRk-ExaMpLEs}). The $(\mu,p)$-Sobolev space $W^{1,p}_\mu(X;\RR^m)$ with respect to the metric measure space $(X,d,\mu)$ is defined as the completion of $\mathcal{A}(X;\RR^m)$ with respect to the norm $\|u\|_{L^p_\mu(X;\RR^m)}+\|\nabla_\mu u\|_{L^p_\mu(X;\MM)}$, where $\MM$ is the space of all $m\times N$ matrices and $\nabla_\mu$, called the $\mu$-gradient, is obtained from $\nabla$ by projection over a suitable ``normal space" to $\mu$ (see \S 2.1 for more details). 

The present paper is a first attempt to establish a general framework to deal with the problem of representing $\overline{E}$ in the setting of metric measure spaces having in mind applications to hyperelasticity. In fact, the interest of considering a general measure is that its support can modeled a hyperelastic structure together with its singularities like for example thin dimensions, corners, junctions, etc.   Such mechanical singular objects naturally lead to develop calculus of variations with metric Sobolev spaces. 

In this paper, we find under which conditions the ``relaxed" variational functional $\overline{E}$ has an integral representation of the form
\begin{equation}\label{RePREsenTAtioN-ProBLEm}
\overline{E}(u;A)=\int_A \overline{L}_x(\nabla_\mu u(x))d\mu(x)
\end{equation}
for all $u\in W^{1,p}_\mu(X;\RR^m)$ and all $A\in\mathcal{O}(X)$ with $\overline{L}_x:T^m_\mu(x)\to[0,\infty]$, where $T^m_\mu(x)$ is the $m$-tangent space to $\mu$ at $x$, i.e., $\MM=T_\mu^m(x)\oplus^\perp N_\mu^m(x)$ with $N_\mu^m(x)$ being the $m$-normal space to $\mu$ at $x$ mentioned above.  We also find a representation formula for $\overline{L}_x$. 

In the setting of euclidean measure spaces, i.e., when $X$ is the closure of a bounded open subset of $\RR^N$, such representation problems was studied, in the one hand, in the convex case in \cite{boubusep97,oah-jpm03,chiadopiat-zhikov03,oah-jpm04}, and, on the other hand, in the non-convex case in \cite{jpm00,jpm05} when $\mu$ is a ``superficial" measure restricted to a smooth manifold. Note also that the study of the lower semicontinuity of variational integrals of type \eqref{RePREsenTAtioN-ProBLEm} was treated in \cite{fragala03} (see also \cite{mocanu05}). In the present paper we prove the following two main integral representation results which extend and complete these previous works to the setting of metric measure spaces both in the convex and non-convex case. 

Firstly, in the convex case, i.e., when  the functions $\widehat{L}_x:T_\mu^m(x)\to[0,\infty]$ given by
$$
\widehat{L}_x(\xi):=\inf_{\zeta\in N_\mu^m(x)}L_x(\xi+\zeta)
$$
are convex, we prove that \eqref{RePREsenTAtioN-ProBLEm} holds with $\overline{L}_x=\widehat{L}_x$ (see Theorem \ref{Main-Theorem-Convex}). Secondly, in the non-convex case, i.e., when the functions $\widehat{L}_x$ are not necessarily convex, we prove, under suitable conditions on the metric measure space $(X,d,\mu)$ and the $(\mu,p)$-Sobolev space $W^{1,p}_\mu(X;\RR^m)$, see {\rm (C$_0$)}, {\rm (C$_1$)}, {\rm (C$_2$)}, {\rm (A$_1$)}, {\rm (A$_2$)} and {\rm (A$_3$)} in \S 2.2.2, that  \eqref{RePREsenTAtioN-ProBLEm} holds with $\overline{L}_x=\mathcal{Q}_\mu L_x$, where $\mathcal{Q}_\mu L_x:T^m_\mu(x)\to[0,\infty]$ is given by
$$
\mathcal{Q}_\mu L_x(\xi):=\lim_{\rho\to 0}\inf\left\{\mint_{Q_\rho(x)}\widehat{L}_y(\xi+\nabla_\mu w(y))d\mu(y):w\in W^{1,p}_{\mu,0}(Q_\rho(x);\RR^m)\right\},
$$
where $W^{1,p}_{\mu,0}(Q_\rho(x);\RR^m)$ is the closure of $\mathcal{A}_0(Q_\rho(x);\RR^m)$ with respect to the $W^{1,p}_\mu$-norm with $\mathcal{A}_0(Q_\rho(x);\RR^m):=\{u\in\mathcal{A}(X;\RR^m):u=0\hbox{ on }X\setminus Q_\rho(x)\}$ and  $Q_\rho(x)$ is the open ball centered at $x\in X$ with radius $\rho>0$  (see Theorem \ref{FiNaL-Main-TheOrEM}). According to the classical theory of relaxation, we can say that this formula plays the role of the classical Dacorogna's quasiconvexification formula in the euclidean Lebesgue setting (see \cite{dacorogna08} for more details). It is then natural to call $\{Q_\mu L_x\}$ the $\mu$-quasiconvexification (or the quasiconvexification with respect to $\mu$) of $\{L_x\}$. 

\medskip

The plan of the paper is as follows. In \S 2.1, Sobolev spaces with respect to a metric measure space are introduced by using the notion of ``normal and tangent space" to a measure as developped in \cite[\S 2]{boubusep97}, \cite[\S 7]{oah-jpm03} and \cite[\S 2]{jpm05} (see also \cite{zhikov96,zhikov00}) in the setting of euclidean measure spaces. In \S 2.2, we state the main results of the paper, i.e., Theorem \ref{Main-Theorem-Convex} in \S 2.2.1 for the convex case and Theorems \ref{N-C-IRT}, \ref{IR-Theorem} and \ref{FiNaL-Main-TheOrEM} in \S 2.2.2 for the non-convex case. These theorems can be applied in the setting of euclidean measure spaces mentioned above, but also in that of (non-euclidean) metric measure spaces endowed with Cheeger-Keith's differentiable structure (see \S 2.3, Corollaries \ref{coro-Cheeger-1} and \ref{coro-Cheeger-2}) whose examples are Carnot groups, glued spaces, Laakso spaces, Bourdon-Pajot spaces and Gromov-Hausdorff limit spaces (see \cite{cheeger99, hajlaszkoskela00, gromov07} and the references therein). In Section 3, we recall two results, i.e., an interchange theorem of infimum and integral, see Theorem \ref{Interchange-Theorem}, and De Giorgi-Letta's lemma, see Lemma \ref{DeGiorgi-Letta-Lemma}, that we use in Section 4 to prove the main results of the paper. The interchange theorem is the principal ingredient in the proof of Proposition \ref{MR1} in \S 4.1, which is used, in the one hand,  to prove Theorem \ref{Main-Theorem-Convex} in \S 4.2, and, in the other hand, to establish, together with De Giorgi's slicing method, a more useful ``relaxed" formula for the variational functional $\overline{E}$, see Lemma \ref{AddiTiONaL-LEMma} (see also Lemma \ref{LeMMa-Main-Convex-Theorem}). De Giorgi-Letta's lemma combined with De Giorgi's slicing method are the essential tools in the proof of Theorem \ref{N-C-IRT} in \S 4.3. Theorem \ref{FiNaL-Main-TheOrEM} is established in \S 4.5 by using again De Giorgi's slicing method together with Theorem \ref{IR-Theorem} whose proof, given in \S4.4, is adapted from \cite[Lemmas 3.3 and 3.5]{boufonmas98} and uses Lemma \ref{AddiTiONaL-LEMma} and Theorem \ref{N-C-IRT}.

\medskip

\subsubsection*{Some basic notation} The open and closed balls centered at $x\in X$ with radius $\rho>0$ are denoted by:
\begin{itemize}
\item $Q_\rho(x):=\Big\{y\in X:d(x,y)<\rho\Big\};$
\item $\overline{Q}_\rho(x):=\Big\{y\in X:d(x,y)\leq\rho\Big\}.$
\end{itemize}
For $x\in X$ and $\rho>0$ we set 
$$
\partial Q_\rho(x):=\overline{Q}_\rho(x)\setminus Q_\rho(x)=\Big\{y\in X:d(x,y)=\rho\Big\}.
$$
For $A\subset X$ and $\eps>0$ we set:
\begin{itemize}
\item $A^-(\eps):=\Big\{x\in X:{\rm dist}(x,A)\leq\eps\Big\};$
\item $A^+(\eps):=\Big\{x\in X:{\rm dist}(x,A)\geq\eps\Big\},$
\end{itemize}
where ${\rm dist}(x,A):=\inf\limits_{a\in A}d(x,a)$. The symbol $\displaystyle \mint$ stands for the mean-value integral 
$$
\mint_Bf d\mu={1\over\mu(B)}\int_Bf d\mu.
$$
 
\section{Main results} 

\subsection{Sobolev spaces with respect to a metric measure space} Let $(X,d)$ be a separable and compact metric space and let $\mu$ be a positive Radon measure on $X$. Let $C(X)$ be the algebra of all continuous functions from $X$ to $\RR$ and let $\mathcal{A}(X)$ be a subalgebra of $C(X)$ such that $1\in\mathcal{A}(X)$. We assume that $\mathcal{A}(X)$ satisfies the Uryshon property, i.e., {\em for every $K\subset V\subset X$ with $K$ compact and $V$ open, there exists $\varphi\in\mathcal{A}(X)$ such that $\varphi(x)\in[0,1]$ for all $x\in X$, $\varphi(x)=0$ for all $x\in X\setminus V$ and $\varphi(x)=1$ for all $x\in K$.} Such a function $\varphi\in\mathcal{A}(X)$ is called a Uryshon function for the pair $(X\setminus V,K)$.

\begin{remark}\label{ReMaRk-ExaMpLEs}
For $(X,d)\equiv(\overline{\Omega},|\cdot-\cdot|)$ where $\Omega$ is a bounded open subset of $\RR^N$ and $|\cdot|$ is the norm in $\RR^N$, the set $C^1(\overline{\Omega})$ (of all  restrictions to $\overline{\Omega}$ of $C^1$-functions from $\RR^N$ to $\RR$ with compact support) is a subalgebra of $C(\overline{\Omega})$ which contains $1$ and satisfies the Uryshon property. More generally, the set ${\rm Lip}(X)$ of all Lipschitz functions from $X$ to $\RR$ is a subalgebra of $C(X)$ containing $1$ and verifying the Uryshon property.
\end{remark}

Denote the class of all subsets $K$ of $X$ such that either $K=A^i(\eps)$, with $A$ an open subset of $X$, $\eps>0$ and $i\in\{-,+\}$, or $K=\overline{Q}_\rho(x)$, with $x\in X$, $\rho>0$ and $\mu(\partial Q_\rho(x))=0$, by $\mathcal{K}(X)$. Let $N\geq 1$ be an integer and let $D:\mathcal{A}(X)\to L_\mu^\infty(X;\RR^N)$ be a linear operator such that:
\begin{eqnarray}
&&D(fg)=fDg+gDf\hbox{ for all }f,g\in\mathcal{A}(X);\label{Der-Prop}\\
&&\hbox{for every }f\in\mathcal{A}(X),\hbox{ every }K\in \mathcal{K}(X)\hbox{ and every }c\in\RR,\label{Der-Prop-1}\\ &&\hbox{if }f(x)=c\hbox{ for all }x\in K\hbox{ then }Df(x)=0\hbox{ for }\mu\hbox{-a.a. } x\in K.\nonumber
\end{eqnarray}
Let $m\geq 1$ be an integer, let $\MM$ be the space of all real $m\times N$ matrices, let $\mathcal{A}(X;\RR^m):=[\mathcal{A}(X)]^m$ and let $\nabla:\mathcal{A}(X;\RR^m)\to L^\infty_\mu(X;\MM)$ be the linear operator given by
\begin{equation}\label{vectorial-D}
\nabla u:=\left(
\begin{array}{c}
Du_1\\
\vdots\\ 
Du_m
\end{array}
\right)
\hbox{ with }u=(u_1,\cdots,u_m).
\end{equation}
Taking \eqref{Der-Prop}  into account it is easy to see that 
\begin{eqnarray}
&&\nabla (f u)=f\nabla u+D f\otimes u\hbox{ for all }u\in\mathcal{A}(X;\RR^m)\hbox{ and all }f\in\mathcal{A}(X).\label{Gradient-Property}
\end{eqnarray}
(Note that $\nabla\equiv D$ when $m=1$.) For each $u\in\mathcal{A}(X;\RR^m)$, set 
\begin{equation}\label{Def-of-Amu}
\mathcal{A}_u^m:=\Big\{v\in\mathcal{A}(X;\RR^m):v(x)=u(x)\hbox{ for all }x\in\spt(\mu) \Big\},
\end{equation}
where $\spt(\mu)$ denotes the support of the measure $\mu$, i.e., $\spt(\mu)$ is the smallest closed set $F\subset X$ such that $\mu(X\setminus\ F)=0$, and consider $\mathcal{H}_u^m$ defined by
$$
\mathcal{H}_u^m:=\Big\{w\in L_\mu^\infty(X;\MM):w(x)=\nabla v(x)\hbox{ for $\mu$-a.a. }x\in X\hbox{ with }v\in\mathcal{A}_u^m\Big\}.
$$
(Note that $\mathcal{A}^m_u\equiv[\mathcal{A}^1_u]^m$ and  $\mathcal{H}^m_u\equiv[\mathcal{H}^1_u]^m$.)
Noticing that $\mathcal{H}_0^m$ (which corresponds to $\mathcal{H}_u^m$ with $u=0$) is a linear subspace of $L_\mu^\infty(X;\MM)$, for $\mu$-a.e. $x\in X$, we introduce $N_\mu^m(x)\subset \MM$ given by
$$
N_\mu^m(x):=\Big\{w(x):w\in\mathcal{H}_0^m\Big\}.
$$

\begin{remark}
In fact, $N_\mu^m(x)=\{\nabla v(x):v\in\mathcal{A}_0^m\}$ for $\mu$-a.a. $x\in A$, where $\mathcal{A}_0^m$ corresponds to $\mathcal{A}_u^m$ with $u=0$. The sets $\mathcal{H}_u^m$ will be useful in the proof of Proposition \ref{MR1} in \S 4.1.
\end{remark}

Then, for $\mu$-a.e. $x\in X$, $N_\mu^m(x)$ is a linear subspace of $\MM$ that we call the $m$-normal space to $\mu$ at $x$. For $\mu$-a.e. $x\in X$, the linear subspace $T_\mu^m(x)$ of $\MM$ given by $\MM=T_\mu^m(x)\oplus^\perp N_\mu^m(x)$ is called the $m$-tangent space to $\mu$ at $x$ and the orthogonal projection on $T^m_\mu(x)$ is denoted by $P_\mu^m(x):\MM\to T_\mu^m(x)$. (Note that $N^m_\mu(x)\equiv[N^1_\mu(x)]^m$ and $T^m_\mu(x)\equiv[T^1_\mu(x)]^m$.) 

Taking \eqref{Gradient-Property} and \eqref{Der-Prop-1} into account we see that the linear operator $\nabla_\mu:\mathcal{A}(X;\RR^m)\to L^\infty_\mu(X;\MM)$ defined, for $\mu$-a.e. $x\in X$, by 
\begin{equation}\label{mu-vectorial-D}
\nabla_\mu u(x):=P_\mu^m(x)(\nabla u(x))=\left(
\begin{array}{c}
P_\mu^1(x)(Du_1(x))\\
\vdots\\ 
P_\mu^1(x)(Du_m(x))
\end{array}
\right)
\hbox{with } u=(u_1,\cdots,u_m)
\end{equation}
satisfies the following properties:
\begin{eqnarray}
&&\nabla_\mu (f u)=f\nabla_\mu u+D_\mu f\otimes u\hbox{ for all }u\in\mathcal{A}(X;\RR^m)\hbox{ and all }f\in\mathcal{A}(X);\label{MU-Gradient-Property}\\
&&\hbox{for every }f\in\mathcal{A}(X),\hbox{ every }K\in\mathcal{K}(X)\hbox{ and every }c\in\RR,\label{MU-Der-Prop-1}\\ 
&&\hbox{if }f(x)=c\hbox{ for all }x\in K\hbox{ then }D_\mu f(x)=0\hbox{ for }\mu\hbox{-a.a.  }x\in K,\nonumber
\end{eqnarray}
where $D_\mu f$ corresponds to $\nabla_\mu f$ with $m=1$.  Moreover, we have
\begin{lemma}\label{comp-EquA-mu-aa}
The linear operator $\nabla_\mu$ is compatible with the equality $\mu$-a.e., i.e.,
\begin{equation}\label{Lemma-Compatible-with-mu-a.e.}
  \hbox{if }u\in \mathcal{A}(X;\RR^m)\hbox{ and if } v\in\mathcal{A}_u^m\hbox{ then } \nabla_\mu u(x)=\nabla_\mu v(x) \hbox { for $\mu$-a.a. $x\in X$.}
\end{equation}
\end{lemma}
\begin{proof}
If $u\in \mathcal{A}(X;\RR^m)$ and if $v\in\mathcal{A}^m_u$ then, for $\mu$-a.e. $x\in X$, $\nabla (u-v)(x)\in N_\mu^m(x)$ and so $P^m_\mu(x)(\nabla (u-v)(x))=0$. Noticing that $\nabla u=\nabla v+\nabla (u-v)$ it follows that $P_\mu^m(x)(\nabla u(x))=P_\mu^m(x)(\nabla v(x))$ for $\mu$-a.a. $x\in X$, i.e., $\nabla_\mu u(x)=\nabla_\mu v(x)$ for $\mu$-a.a. $x\in X$. 
\end{proof}

\medskip

Let $1\leq p\leq\infty$ be a real number. The $(\mu,p)$-Sobolev space $W^{1,p}_\mu(X;\RR^m)$ with respect to the metric measure space $X=(X,d,\mu)$ is defined as the completion of $\mathcal{A}(X;\RR^m)$ with respect to the norm
\begin{equation}\label{W1pmu-norm}
\|u\|_{W^{1,p}_\mu(X;\RR^m)}:=\|u\|_{L^p_\mu(X;\RR^m)}+\|\nabla_\mu u\|_{L_\mu^p(X;\MM)}.
\end{equation}
Since $\|\nabla_\mu u\|_{L_\mu^p(X;\MM)}\leq \|u\|_{W^{1,p}_\mu(X;\RR^m)}$ for all $u\in\mathcal{A}(X;\RR^m)$ the linear map $\nabla_\mu$ from $\mathcal{A}(X;\RR^m)$ to $L_\mu^p(X;\MM)$  has a unique extension to $W^{1,p}_\mu(X;\RR^m)$ which will still be denoted by $\nabla_\mu$ and will be called the $\mu$-gradient.

\begin{remark}
When $X$ is the closure of a bounded open subset $\Omega$ of $\RR^N$ and $\mu$ is the Lebesgue measure on $\overline{\Omega}$, we retreive the (classical) Sobolev spaces $W^{1,p}(\Omega;\RR^m)$. If $X$  is a compact manifold $M$ and if $\mu$ is the superficial measure on $M$, we obtain the (classical) Sobolev spaces $W^{1,p}(M;\RR^m)$ on the compact manifold $M$. For more details on the various possible extensions of the classical theory of the Sobolev spaces to the setting of metric measure spaces, we refer to \cite[\S 10-14]{heinonen07} (see also \cite{cheeger99, gol-tro01,hajlasz02}).
\end{remark}

\begin{remark}\label{Remark-MU-Gradient-Property-BiS}
As $\mathcal{A}(X)$ is an algebra we have $fu\in\mathcal{A}(X;\RR^m)$ for all $f\in \mathcal{A}(X)$ and all $\mathcal{A}(X;\RR^m)$, and so $fu\in W^{1,p}_\mu(X;\RR^m)$ for all $f\in \mathcal{A}(X)$ and all $W^{1,p}_\mu(X;\RR^m)$ because $\mathcal{A}(X)$ is a subclass of the algebra of all continuous functions from $X$ to $\RR$ and $X$ is a compact metric space. On the other hand, from \eqref{MU-Gradient-Property} we see that
\begin{equation}\label{MU-Gradient-Property-BiS}
\nabla_\mu (f u)=f\nabla_\mu u+D_\mu f\otimes u\hbox{ for all }u\in W^{1,p}_\mu(X;\RR^m)\hbox{ and all }f\in\mathcal{A}(X).
\end{equation}
\end{remark}

\begin{remark}[generalization of Lemma \ref{comp-EquA-mu-aa}]\label{Remark-Generalization-of-Lemmacomp-EquA-mu-aa}
Given $A\in\mathcal{O}(X)$ set:
\begin{itemize}
\item $\mathcal{A}^m_0(A):=\Big\{v\in\mathcal{A}(X;\RR^m):v(x)=0\hbox{ for all } x\in\spt(\mu)\cap A\Big\}$;
\item $N^m_\mu(x,A):=\Big\{\nabla v(x):v\in\mathcal{A}^m_0(A)\Big\}$ for $\mu$-a.a. $x\in A$.
\end{itemize}
The following makes clear the link between $N_\mu^m(x)$ and $N_\mu^m(x,A)$.
\begin{lemma}\label{ReMaRK-GeneRalIzAtION-LeMMa}
$N_\mu^m(x,A)=N_\mu^m(x)$ for $\mu$-a.a. $x\in A$.
\end{lemma}
\begin{proof}
As $\mathcal{H}^m_0\subset\mathcal{H}^m_0(A)$ we have $N_\mu(x)\subset N_\mu(x,A)$ for $\mu$-a.a. $x\in A$. On the other hand, let $\xi\in N_\mu^m(x,A)$. Then $\xi=\nabla v(x)$ with $v\in\mathcal{A}^m_0(A)$. As $A$ is open we have $\overline{Q}_\rho(x)\subset A$ for some $\rho>0$. As $\mathcal{A}(X)$ satisfies the Uryshon property, there exists a Uryshon function $\varphi\in\mathcal{A}(X)$ for the pair $(X\setminus A,\overline{Q}_\rho(x))$. Set $\overline{v}:=\varphi v$. Then $\overline{v}\in\mathcal{A}^m_0$ because $\varphi(y)=0$ for all $y\in X\setminus A$ and $v(y)=0$ for all $y\in\spt(\mu)\cap A$. On the other hand, using \eqref{Gradient-Property} we see that $\nabla \overline{v}=D\varphi\otimes v+\varphi\nabla v$, and so $\nabla \overline{v}(x)=\nabla v(x)$ since $\varphi(x)=1$ and $v (x)=0$. It follows that $\xi\in N^m_\mu(x)$.
\end{proof}

\medskip

The following lemma, which generalizes Lemma \ref{comp-EquA-mu-aa}, is a consequence of Lemma \ref{ReMaRK-GeneRalIzAtION-LeMMa}.

\begin{lemma}
If $v\in\mathcal{A}(X;\RR^m)$ is such that $v(x)=0$ for all $x\in\spt(\mu)\cap A$, then $\nabla_\mu v(x)=0$ for $\mu$-a.a. $x\in A$.
\end{lemma}
\begin{proof}
If $v\in\mathcal{A}(X;\RR^m)$ is such that $v(x)=0$ for all $x\in\spt(\mu)\cap A$, then $v\in\mathcal{A}^m_0(A)$, and so, for $\mu$-a.e. $x\in A$, $\nabla v(x)\in N_\mu^m(x,A)$. But, by Lemma \ref{ReMaRK-GeneRalIzAtION-LeMMa}, $N_\mu^m(x,A)=N_\mu^m(x)$ for $\mu$-a.a. $x\in A$, which means that $\nabla v(x)\in N_\mu^m(x)$ for $\mu$-a.a. $x\in A$. It follows that $P_\mu^m(\nabla v(x))=0$ for $\mu$-a.a. $x\in A$, i.e., $\nabla_\mu v(x)=0$ for $\mu$-a.a. $x\in A$.
\end{proof}
\end{remark}

\subsection{Integral representation theorems} Let $p\in]1,\infty[$ be a real number and let $\{L_x\}$ be a field of Carath\'eodory integrands over $X$, i.e., to $\mu$-a.e. $x\in X$ there corresponds a continuous function $L_x:\MM\to[0,\infty]$ so that the function $x\mapsto L_x(\xi)$ is $\mu$-measurable for all $\xi\in\MM$. We assume that $\{L_x\}$ is $p$-coercive, i.e., there exists $C>0$ such that
\begin{equation}\label{p-coercivity}
L_x(\xi)\geq C|\xi|^p\hbox{ for all }\xi\in\MM\hbox{ and $\mu$-a.a. }x\in X,
\end{equation}
and of $p$-polynomial growth, i.e., there exists $c>0$ such that
\begin{equation}\label{p-polynomial-growth}
L_x(\xi)\leq c(1+|\xi|^p)\hbox{ for all }\xi\in\MM\hbox{ and $\mu$-a.a. }x\in X.
\end{equation}
Let $\mathcal{O}(X)$ be the class of all open subsets of $X$, let $E:\mathcal{A}(X;\RR^m)\times\mathcal{O}(X)\to[0,\infty]$ be the variational integral defined by
$$
E(u;A):=\int_A L_x(\nabla u(x))d\mu(x)
$$
and let $\overline{E}:W^{1,p}_\mu(X;\RR^m)\times\mathcal{O}(X)\to[0,\infty]$ be the ``relaxed" variational functional of the variational integral $E$ with respect to the strong convergence in $L^p_\mu(X;\RR^m)$, i.e.,  
$$
\overline{E}(u;A):=\inf\left\{\liminf_{n\to\infty}E(u_n;A):\mathcal{A}(X;\RR^m)\ni u_n\to u\hbox{ in }{L^p_\mu(X;\RR^m)}\right\}.
$$
Note that the variational integral $E$ is in general not ``local", i.e., $u(x)=v(x)$ for $\mu$-a.a. $x\in X$ does not imply $E(u;A)=E(v;A)$ for all $A\in\mathcal{O}(X)$. However, as it is stated in the following proposition, the variational functional $\overline{E}$ can be rewritten as the ``relaxed" variational functional of a ``local" variational integral depending on the $\mu$-gradient. Let $\widehat{E}:\mathcal{A}(X;\RR^m)\times\mathcal{O}(X)\to[0,\infty]$ be defined by
$$
\widehat{E}(u;A):=\int_A \widehat{L}_x(\nabla_\mu u(x))d\mu(x)
$$
where, for $\mu$-a.e. $x\in X$, $\widehat{L}_x:T_\mu^m(x)\to[0,\infty]$ is given by
$$
\widehat{L}_x(\xi):=\inf_{\zeta\in N_\mu^m(x)}L_x(\xi+\zeta).
$$

\begin{remark}\label{p-coercivity-p-polynomial-growth-remark}
It is easy to see that, in the one hand, if $\{L_x\}$ is $p$-coercive then also is $\{\widehat{L}_x\}$, i.e.,
$$
\widehat{L}_x(\xi)\geq C|\xi|^p\hbox{ for all }\xi\in T^m_\mu(x)\hbox{ and $\mu$-a.a. }x\in X
$$
with $C>0$ given by \eqref{p-coercivity}, and, on the other hand, if $\{L_x\}$ is of $p$-polynomial growth then also is $\{\widehat{L}_x\}$, i.e.,
$$
\widehat{L}_x(\xi)\leq c(1+|\xi|^p)\hbox{ for all }\xi\in T^m_\mu(x)\hbox{ and $\mu$-a.a. }x\in X
$$
with $c>0$ given by \eqref{p-polynomial-growth}.
\end{remark}

\begin{remark}\label{continuity-remaRk}
If $L_x$ is continuous for $\mu$-a.a. $x\in X$ and if \eqref{p-coercivity} holds, i.e., $L_x$ is $p$-coercive, then $\widehat{L}_x$ is continuous for $\mu$-a.a. $x\in X$. Indeed, let $\xi\in T^m_\mu(x)$ and let $\{\xi_i\}_i\subset T^m_\mu(x)$ be such that $|\xi_i-\xi|\to0$. As $L_x$ is continuous and, for every $\zeta\in N_\mu^m(x)$, $\widehat{L}_x(\xi_i)\leq L_x(\xi_i+\zeta)$ for all $i\geq 1$ we have $\limsup_{i\to\infty}\widehat{L}_x(\xi_i)\leq L_x(\xi+\zeta)$ for all $\zeta\in N_\mu^m(x)$, and so $\limsup_{i\to\infty}\widehat{L}_x(\xi_i)\leq \widehat{L}_x(\xi)$. On the other hand, there is no loss of generality in assuming that $\liminf_{i\to\infty}\widehat{L}_x(\xi_i)=\lim_{i\to\infty}\widehat{L}_x(\xi_i)<\infty$. Consider $\{\zeta_i\}_i\subset N^m_\mu(x)$ such that $\widehat{L}_x(\xi_i)\leq L_x(\xi_i+\zeta_i)<\widehat{L}_x(\xi_i)+{1\over i}$ for all $i\geq 1$. As \eqref{p-coercivity} holds we see that the sequence $\{\zeta_i\}_i$ is bounded, and so (up to a subsequence) we can assert that there exists $\zeta\in N^m_\mu(x)$ such that $|\zeta_i-\zeta|\to0$. From the continuity of $L_x$ we deduce that  $\liminf_{i\to\infty}\widehat{L}_x(\xi_{i})=L_x(\xi+\zeta)\geq \widehat{L}_x(\xi)$, and the result follows.
\end{remark}

\begin{proposition}\label{MR1}
If \eqref{p-polynomial-growth} holds then
$$
\overline{E}(u;A):=\inf\left\{\liminf_{n\to\infty}\widehat{E}(u_n;A):\mathcal{A}(X;\RR^m)\ni u_n\to u\hbox{ in }L^p_\mu(X;\RR^m)\right\}
$$
for all $u\in W^{1,p}_\mu(X;\RR^m)$ and all $A\in\mathcal{O}(X)$.
\end{proposition}

\begin{remark}
Taking \eqref{Lemma-Compatible-with-mu-a.e.} into account it is easy to see that the variational integral $\widehat{E}$ is ``local", i.e., if $u(x)=v(x)$ for $\mu$-a.a. $x\in X$ then $\widehat{E}(u;A)=\widehat{E}(v;A)$ for all $A\in\mathcal{O}(X)$. Thus, the variational functional $\widehat{\mathcal{E}}:W^{1,p}_\mu(X;\RR^m)\times\mathcal{O}(X)\to[0,\infty]$ given by
\begin{equation}\label{DeF-fUnCt-RemarK-Local}
\widehat{\mathcal{E}}(u;A):=\left\{
\begin{array}{ll}
\widehat{E}(u;A)&\hbox{if }u\in \mathcal{A}(X;\RR^m)\\
\infty&\hbox{otherwise}
\end{array}
\right.
\end{equation}
is well defined with respect to the equality $\mu$-a.e.. We can then rephrase Proposition \ref{MR1} as follows: {\em the variational functional $\overline{E}$ is the variational lower semicontinuous envelope of $\widehat{\mathcal{E}}$ with respect to the strong convergence in $L^p_\mu(X;\RR^m)$}.
\end{remark}

\begin{remark}\label{ReFLeXIVITY-of-Metric-SObOlev-Spaces}
The $(\mu,p)$-Sobolev space $W^{1,p}_\mu(X;\RR^m)$ is reflexive whenever $p\in]1,\infty[$. Indeed, the linear operator $\Theta:W^{1,p}_\mu(X;\RR^m)\to L^p_\mu(X;\RR^m)\times L^p_\mu(X;\MM)$ defined by $\Theta(u):=(u,\nabla_\mu u)$ is an isometry, hence $\Theta(W^{1,p}_\mu(X;\RR^m))$ is a closed linear subspace of $L^p_\mu(X;\RR^m)\times L^p_\mu(X;\MM)$. For $p>1$ the product space $L^p_\mu(X;\RR^m)\times L^p_\mu(X;\MM)$ is reflexive, and so is $\Theta(W^{1,p}_\mu(X;\RR^m))$.
\end{remark}

\subsubsection{The convex case} The following theorem gives, under \eqref{p-coercivity} and \eqref{p-polynomial-growth}, an integral representation of the ``relaxed" variational functional $\overline{E}$ in the reflexive and convex case. 

\begin{theorem}\label{Main-Theorem-Convex}
If \eqref{p-coercivity} and \eqref{p-polynomial-growth} hold and if $\widehat{L}_x$ is convex for $\mu$-a.a. $x\in X$, then
$$
\overline{E}(u;A)=\int_A \widehat{L}_x(\nabla_\mu u(x))d\mu(x)
$$
for all $u\in W^{1,p}_\mu(X;\RR^m)$ and all $A\in\mathcal{O}(X)$.
\end{theorem}

\begin{remark}
If $L_x$ is convex for $\mu$-a.a. $x\in X$  then also is $\widehat{L}_x$ for $\mu$-a.a. $x\in X$. Indeed, let $\alpha\in]0,1[$ and let $\xi,\hat\xi\in T^m_\mu(x)$ and consider  $\{\zeta_i\}_i,\{\hat\zeta_i\}_i\subset N^m_\mu(x)$ such that $\widehat{L}_x(\xi)=\lim_{i\to\infty}L_x(\xi+\zeta_i)$ and $\widehat{L}_x(\hat\xi)=\lim_{i\to\infty}L_x(\hat\xi+\hat\zeta_i)$. Fix any $i\geq 1$. As $\alpha\zeta_i+(1-\alpha)\hat\zeta_i\in N^m_\mu(x)$ we have $\widehat{L}_x(\alpha\xi+(1-\alpha)\hat\xi)\leq L_x(\alpha\xi+(1-\alpha)\hat\xi+\alpha\zeta_i+(1-\alpha)\hat\zeta_i)=L_x(\alpha(\xi+\zeta_i)
+(1-\alpha)(\hat\xi+\hat\zeta_i))$. Hence $\widehat{L}_x(\alpha\xi+(1-\alpha)\hat\xi)\leq\alpha L_x(\xi+\zeta_i)+(1-\alpha)L_x(\hat\xi+\hat\zeta_i)$ for all $i\geq 1$ because $L_x$ is convex, and the result follows by letting $i\to\infty$.

However, the converse implication is not true. Indeed, if for $\mu$-a.e. $x\in X$, $L_x:\MM\to[0,\infty]$ is of the form
$$
L_x(\xi)=L_1\big(P^m_\mu(x)(\xi)\big)+L_2\big(\xi-P^m_\mu(x)(\xi)\big),
$$
with $L_1,L_2:\MM\to[0,\infty]$ such that $L_1$ is convex and $L_2$ is not convex, then both $L_x$ is not convex and $\widehat{L}_x:T^m_\mu(x)\to[0,\infty]$ is convex.
\end{remark}

\subsubsection{The non-convex case} In the non-convex case, i.e., when the functions $\widehat{L}_x$ are not necessarily convex, the following theorem asserts that under \eqref{p-coercivity} and \eqref{p-polynomial-growth} the variational functional $\overline{E}$ has always a ``general" integral representation.

\begin{theorem}\label{N-C-IRT}
If \eqref{p-coercivity} and \eqref{p-polynomial-growth} hold then  
$$
\overline{E}(u;A)=\int_A \lambda_u(x)d\mu(x)
$$
for all $u\in W^{1,p}_\mu(X;\RR^m)$ and all $A\in\mathcal{O}(X)$ with $\lambda_u\in L^1_\mu(X)$ given by 
$$
\lambda_u(x):=\lim_{\rho\to0}{\overline{E}(u;Q_\rho(x))\over\mu(Q_\rho(x))}.
$$ 
\end{theorem}

To refine the ``general" integral representation given by Theorem \ref{N-C-IRT}, we need the following four conditions:
\begin{itemize}
\item[(C$_0$)] the $\mu$-gradient is closable in $W^{1,p}_\mu(X;\RR^m)$, i.e., for every $u\in W^{1,p}_\mu(X;\RR^m)$ and every $A\in\mathcal{O}(X)$, if $u(x)=0$ for $\mu$-a.a. $x\in A$ then $\nabla_\mu u(x)=0$ for $\mu$-a.a. $x\in A$;
\item[(C$_1$)] $X$ supports a $p$-Sobolev inequality, i.e., there exist $K>0$ and $\chi\geq 1$ such that
\begin{equation}\label{Poincare-Inequality}
\left(\int_{Q_\rho(x)}|v|^{\chi p}d\mu\right)^{1\over\chi p}\leq \rho K\left(\int_{Q_\rho(x)}|\nabla_\mu v|^pd\mu\right)^{1\over p}
\end{equation}
for all $0<\rho\leq \rho_0$, with $\rho_0>0$, and all $v\in W^{1,p}_{\mu,0}(Q_\rho(x);\RR^m)$, where, for each $A\in\mathcal{O}(X)$, $W^{1,p}_{\mu,0}(A;\RR^m)$ is the closure of $\mathcal{A}_0(A;\RR^m)$ with respect to $W^{1,p}_\mu$-norm defined in \eqref{W1pmu-norm} with $\mathcal{A}_0(A;\RR^m):=\{u\in\mathcal{A}(X;\RR^m):u=0\hbox{ on }X\setminus A\}$;
\item[(C$_2$)] $X$ satisfies the Vitali covering theorem, i.e., for every $A\subset X$ and every family $\mathcal{F}$ of closed balls in X, if $\inf\{\rho>0:\overline{Q}_\rho(x)\in\mathcal{F}\}=0$ for all $x\in A$ then there exists a countable disjointed subfamily $\mathcal{G}$ of $\mathcal{F}$ such that $\mu(A\setminus \cup_{Q\in\mathcal{G}}Q)=0$ (in other words, $A\subset \big(\cup_{Q\in\mathcal{G}}Q\big)\cup N$ with $\mu(N)=0$).
\end{itemize}

\begin{remark}
From Remark \ref{Remark-Generalization-of-Lemmacomp-EquA-mu-aa} we see that the $\mu$-gradient is closable in $\mathcal{A}(X;\RR^m)$. The assumption (C$_0$) asserts that the closability of the $\mu$-gradient can be extended from $\mathcal{A}(X;\RR^m)$ to $W^{1,p}_\mu(X;\RR^m)$. 
\end{remark}

\begin{remark}\label{ReMArK-VItALi-For-OpEN-SEtS}
As $\mu$ is a Radon measure, if $X$ satisfies the Vitali covering theorem, i.e., (C$_2$) holds, then for every $A\in\mathcal{O}(X)$ and every $\eps>0$ there exists a countable family $\{Q_{\rho_i}(x_i)\}_{i\in I}$ of disjoint open balls of $A$ with $x_i\in A$, $\rho_i\in]0,\eps[$ and $\mu(\partial Q_{\rho_i}(x_i))=0$ such that $\mu\big(A\setminus\cup_{i\in I}Q_{\rho_i}(x_i)\big)=0$.
\end{remark}

\begin{theorem}\label{IR-Theorem}
Under \eqref{p-coercivity} and \eqref{p-polynomial-growth}, if {\rm (C$_0$)}, {\rm (C$_1$)} and {\rm (C$_2$)} hold, then 
\begin{eqnarray*}
\lambda_u(x)&=&\lim_{\rho\to 0}\inf\left\{\mint_{Q_\rho(x)}\widehat{L}_y(\nabla_\mu v(y))d\mu(y):v-u\in W^{1,p}_{\mu,0}(Q_\rho(x);\RR^m)\right\}\\
&=&\lim_{\rho\to 0}\inf\left\{\mint_{Q_\rho(x)}\widehat{L}_y(\nabla_\mu u(y)+\nabla_\mu w(y))d\mu(y):w\in W^{1,p}_{\mu,0}(Q_\rho(x);\RR^m)\right\}
\end{eqnarray*}
for all $u\in W^{1,p}_\mu(X;\RR^m)$ and $\mu$-a.a. $x\in X$. 
\end{theorem}

In order to ``localize in $\xi$" the density formula given by Theorem \ref{IR-Theorem} we need to consider the three assumptions below.
\begin{itemize}
\item[(A$_1$)] For every $u\in W^{1,p}_\mu(X;\RR^m)$ and $\mu$-a.e. $x\in X$ there exists $u_x\in W^{1,p}_\mu(X;\RR^m)$ such that:
\begin{eqnarray}
&&\nabla_\mu u_x(y)=\nabla_\mu u(x)\hbox{ for $\mu$-a.a. $y\in X$};\label{FinALAssuMpTIOnOne}\\
&&\lim_{\rho\to0}{1\over\rho^p}\mint_{Q_\rho(x)}|u(y)-u_x(y)|^pd\mu(y)=0.\label{FinALAssuMpTIOnTwo}
\end{eqnarray}
\item[(A$_2$)] For every $x\in X$, every $\rho>0$ and every $t\in]0,1[$ there exists a Uryshon function $\varphi\in\mathcal{A}(X)$ for the pair $(X\setminus Q_\rho(x),\overline{Q}_{t\rho}(x))$ such that 
$$
\|D_\mu\varphi\|_{L^\infty_\mu(X;\RR^N)}\leq{\alpha\over\rho(1-t)}
$$
for some $\alpha>0$.
\item[(A$_3$)] The measure $\mu$ is doubling, i.e., ${\mu(Q_\rho(x))\leq\beta\mu(Q_{\rho\over 2}(x))}$ for some $\beta\geq 1$, all $\rho>0$ and all $x\in X$. (In particular, $X$ satisfies the Vitali covering theorem, i.e., (C$_2$) holds.) We futhermore assume that for $\mu$-a.e. $x\in X$,
\begin{equation}\label{DoublINgAssUMpTiON}
\lim_{t\to 1^-}\limsup_{\rho\to0}{\mu(Q_{t\rho}(x)\over\mu(Q_\rho(x))}=1.
\end{equation}
\end{itemize} 

\begin{remark}
If there is $\theta:]0,1[\to[1,\infty[$ with $\lim_{t\to 1}\theta(t)=1$ such that 
$
\mu(Q_\rho(x))\leq\theta(t)\mu(Q_{t\rho}(x))
$
for all $\rho>0$, all $x\in X$ and all $t\in]0,1[$, then (A$_3$) holds.
\end{remark}

\begin{theorem}\label{FiNaL-Main-TheOrEM}
Under \eqref{p-coercivity} and \eqref{p-polynomial-growth}, if {\rm (C$_0$)}, {\rm (C$_1$)}, {\rm (A$_1$)}, {\rm (A$_2$)} and {\rm (A$_3$)} hold, then 
$$
\overline{E}(u;A)=\int_A \mathcal{Q}_\mu L_x(\nabla_\mu u(x))d\mu(x) 
$$
for all $u\in W^{1,p}_\mu(X;\RR^m)$ and all $A\in\mathcal{O}(X)$ with $\mathcal{Q}_\mu L_x:T_\mu^m(x)\to[0,\infty]$ given by
$$
\mathcal{Q}_\mu L_x(\xi):=\lim_{\rho\to 0}\inf\left\{\mint_{Q_\rho(x)}\widehat{L}_y(\xi+\nabla_\mu w(y))d\mu(y):w\in W^{1,p}_{\mu,0}(Q_\rho(x);\RR^m)\right\}.
$$
\end{theorem}

\begin{remark}
According to the classical theory of relaxation, we can say that this formula plays the role of the classical Dacorogna's quasiconvexification formula in the euclidean Lebesgue setting (see \cite{dacorogna08} for more details). It is then natural to call $\{Q_\mu L_x\}$ the $\mu$-quasiconvexification (or the quasiconvexification with respect to $\mu$) of $\{L_x\}$. 
\end{remark}

\subsection{Application to the setting of Cheeger-Keith's differentiable structure} We begin with the concept of upper gradient introduced by Heinonen and Koskela (see \cite{heinonen-koskela98}). 
\begin{definition}
A Borel function $g:X\to[0,\infty]$ is said to be an upper gradient for $f:X\to\RR$  if 
$
|f(c(1))-f(c(0))|\leq\int_0^1 g(c(s))ds
$
for all continuous rectifiable curves $c:[0,1]\to X$. 
\end{definition}
The concept of upper gradient has been generalized by Cheeger as follows (see \cite[Definition 2.8]{cheeger99}). 
\begin{definition}
A function $g\in L^p_\mu(X)$, with $1<p<\infty$, is said to be a $p$-weak upper gradient for $f\in L^p_\mu(X)$ if there exist $\{f_n\}_n\subset L^p_\mu(X)$ and $\{g_n\}_n\subset L^p_\mu(X)$ such that for each $n\geq 1$, $g_n$ is an upper gradient for $f_n$, $f_n\to f$ in $L^p_\mu(X)$ and $g_n\to g$ in $L^p_\mu(X)$. 
\end{definition}
From Cheeger and Keith (see \cite[Theorem 4.38]{cheeger99} and \cite[Definition 2.1.1 and Theorem 2.3.1]{keith1-04}) we have
\begin{theorem}\label{cheeger-theorem}
Assume that $\mu$ is doubling and $X$ supports a weak $(1,p)$-Poincar\'e inequality with $1<p<\infty$, i.e., there exist $C>0$ and $\sigma\geq 1$ such that for every $\rho>0$, every $f\in L^p_\mu(X)$ and every $p$-weak upper gradient $g\in L^p_\mu(X)$ for $f$,
$$
\mint_{Q_\rho(x)}\left|f-\mint_Q f d\mu\right|d\mu\leq \rho C\left(\mint_{Q_{\sigma \rho}(x)} g^p d\mu\right)^{1\over p}.
$$
Then, there exists a countable family $\{(X_\alpha,\xi^\alpha)\}_\alpha$ of $\mu$-measurable disjoint subsets $X_\alpha$ of $X$ with $\mu(X\setminus\cup_\alpha X_\alpha)=0$ and of functions $\xi^\alpha=(\xi^\alpha_1,\cdots,\xi^\alpha_{N(\alpha)}):X\to\RR^{N(\alpha)}$ with $\xi^\alpha_i\in {\rm Lip}(X)$ satisfying the following properties{\rm:}
\begin{enumerate}[label={\rm(\alph*)}]
\item  there exists an integer $N\geq 1$ such that $N(\alpha)\in\{1,\cdots, N\}$ for all $\alpha;$
\item for every $\alpha$ and every $f\in{\rm Lip}(X)$ there is a unique $D^\alpha f\in L^\infty_\mu(X_\alpha;\RR^{N(\alpha)})$ such that for $\mu$-a.e. $x\in X_\alpha$,
$$
\lim_{\rho\to 0}{1\over\rho}\|f-f_x\|_{L^\infty_{\mu(Q_\rho(x))}}=0,
$$
where $f_x\in {\rm Lip}(X)$ is given by $f_x(y):=f(x)+D^\alpha f(x)\cdot(\xi^\alpha(y)-\xi^\alpha(x));$ in particular $D^\alpha f_x(y)=D^\alpha f(x)$ for $\mu$-a.a. $y\in X_\alpha;$
\item the operator $D:{\rm Lip}(X)\to L^\infty_\mu(X;\RR^N)$ given by
$$
Df:=\sum_\alpha \mathds{1}_{X_\alpha}D^\alpha f,
$$
where $\mathds{1}_{X_\alpha}$ denotes the characteristic function of $X_\alpha$, is linear and, for each $f,g\in{\rm Lip}(X)$, $D(fg)=fDg+gDf;$
\item for every $f\in{\rm Lip}(X)$, $Df=0$ $\mu$-a.e. on every $\mu$-measurable set where $f$ is constant.\end{enumerate}
\end{theorem}

Let ${\rm Lip}(X;\RR^m):=[{\rm Lip(X)}]^m$ and let $\nabla:{\rm Lip}(X;\RR^m)\to L^\infty_\mu(X;\MM)$ given by \eqref{vectorial-D}. From Theorem \ref{cheeger-theorem}(d) we see that $\nabla_\mu\equiv\nabla$, where $\nabla_\mu$ is defined by \eqref{mu-vectorial-D}. (In fact, by Theorem \ref{cheeger-theorem}(d), for $\mu$-a.e. $x\in X$, we have $N_\mu^m(x)\equiv\{0\}$ and so $T_\mu^m(x)\equiv\mathbb{M}^{m\times N}$. In particular $\widehat{L}_x\equiv L_x$.) 

\begin{remark}
In the euclidean setting, i.e., when $X\equiv\overline{\Omega}$, where $\Omega$ is a bounded open subset of $\RR^N$, $\mathcal{A}(X;\RR^m)\equiv C^1(\overline{\Omega};\RR^m)$ and $\nabla$ is the classical gradient, we have $\widehat{L}_x\not\equiv L_x$ whenever $\mu$ is not absolutely continuous with respect to the Lebesgue measure.
\end{remark}

The $(\mu,p)$-Sobolev space $W^{1,p}_\mu(X;\RR^m)$ obtained as the closure of ${\rm Lip}(X;\RR^m)$ with respect to the $W^{1,p}_\mu$-norm defined in \eqref{W1pmu-norm} is called the $p$-Cheeger-Sobolev space. In this framework, from Theorem \ref{Main-Theorem-Convex} we obtain

\begin{corollary}\label{coro-Cheeger-1}
Under the hypotheses of Theorem {\rm\ref{cheeger-theorem}}, if  \eqref{p-coercivity} and \eqref{p-polynomial-growth} hold and if $L_x$ is convex for $\mu$-a.a. $x\in X$, then
$$
\overline{E}(u;A)=\int_A L_x(\nabla_\mu u(x))d\mu(x)
$$
for all $u\in W^{1,p}_\mu(X;\RR^m)$ and all $A\in\mathcal{O}(X)$.
\end{corollary} 

For the non-convex case, we have

\begin{proposition}\label{Prop-Cheeger}
Under the hypotheses of Theorem {\rm\ref{cheeger-theorem}}, the assumptions {\rm(C$_0$)}, {\rm(C$_1$)}, {\rm(C$_2$)}, {\rm(A$_1$)} and {\rm(A$_2$)} hold. If moreover $(X,d)$ is a length space then \eqref{DoublINgAssUMpTiON} is also satisfied, i.e., {\rm (A$_3$)} holds. 
\end{proposition}
\begin{proof}
Firstly, since $\mu$ is doubling, $X$ satisfies the Vitali covering theorem, i.e., (C$_2$) holds. Secondly, the closability of the $\mu$-gradient in ${\rm Lip}(X;\RR^m)$, given by Theorem \ref{cheeger-theorem}(d), can be extended from ${\rm Lip}(X;\RR^m)$ to $W^{1,p}_\mu(X;\RR^m)$ by using the closability theorem of Franchi, Haj{\l}asz and Koskela (see \cite[Theorem 10]{fran-haj-kos99}). Thus, (C$_0$) is satisfied. Thirdly,  according to Cheeger (see \cite[\S 4, p. 450]{cheeger99} and also \cite{Haj-Kos-95,hajlaszkoskela00}), since $\mu$ is doubling and $X$ supports a weak $(1,p)$-Poincar\'e inequality, we can assert that $X$ supports a $p$-Sobolev inequality, i.e., there exist $c>0$ and $\chi>1$ such that for every $0<\rho\leq\rho_0$, with $\rho_0\geq 0$, every $v\in W^{1,p}_{\mu,0}(X;\RR^m)$ and every $p$-weak upper gradient $g\in L^p_\mu(X;\RR^m)$ for $v$,
\begin{equation}\label{Cheeger-Coro-Eq1}
\left(\int_{Q_\rho(x)}|v|^{\chi p}d\mu\right)^{1\over\chi p}\leq \rho c\left(\int_{Q_\rho(x)}|g|^pd\mu\right)^{1\over p}.
\end{equation}
On the other hand, from Cheeger (see \cite[Theorems 2.10 and 2.18]{cheeger99}), for each $w\in W^{1,p}_\mu(X)$ there exists a unique $p$-weak upper gradient for $w$, denoted by $g_w\in L ^p_\mu(X)$ and called the minimal $p$-weak upper gradient for $w$, such that for every  $p$-weak upper gradient $g\in L^p_\mu(X)$ for $w$, $g_w(x)\leq g(x)$ for $\mu$-a.a.  $x\in X$. Moreover (see \cite[\S 4]{cheeger99} and also \cite[\S B.2, p. 363]{bjorn-bjorn-11}, \cite{bjorn00} and \cite[Remark 2.15]{gong-haj-13}), there exists $\alpha\geq 1$ such that for every $w\in W^{1,p}_\mu(X)$ and $\mu$-a.e. $x\in X$, 
$$
{1\over\alpha} |g_w(x)|\leq|D_\mu w(x)|\leq\alpha|g_w(x)|,
$$ 
where $D_\mu$ corresponds to $\nabla_\mu$ with $m=1$. As for $v=(v_i)_{i=1,\cdots,m}\in W^{1,p}_\mu(X;\RR^m)$ we have $\nabla_\mu v=(D_\mu v_i)_{i=1,\cdots,m}$, it follows that 
\begin{equation}\label{Cheeger-Coro-Eq2}
{1\over \alpha} |g_v(x)|\leq|\nabla_\mu v(x)|\leq\alpha|g_v(x)|
\end{equation}
for $\mu$-a.a. $x\in X$, where $g_v:=(g_{v_i})_{i=1,\cdots,m}$ is naturally called the minimal $p$-weak upper gradient for $v$. Combining \eqref{Cheeger-Coro-Eq1} with \eqref{Cheeger-Coro-Eq2} we obtain (C$_1$). Fourthly,  from Bj{\"o}rn (see \cite[Theorem 4.5 and Corollary 4.6]{bjorn00} and also \cite[Theorem 2.12]{gong-haj-13}) we see that for every $\alpha$, every $u\in W^{1,p}_\mu(X;\RR^m)$ and $\mu$-a.e. $x\in X_\alpha$,
$$
\nabla_\mu u_x(y)=\nabla_\mu u(x)\hbox{ for }\mu\hbox{-a.a. }y\in X_\alpha,
$$
where $u_x\in W^{1,p}_\mu(X;\RR^m)$ is given by 
$$
u_x(y):=u(y)-u(x)-\nabla_\mu u(x)\cdot(\xi^\alpha(y)-\xi^\alpha(x))
$$
 and $u$ is $L^p_\mu$-differentiable at $x$, i.e.,
$$
\lim_{\rho\to 0}{1\over\rho}\|u(y)-u_x(y)\|_{L^p_\mu(Q_\rho(x);\RR^m)}=0.
$$
Hence (A$_1$) is verified. Fifthly, given $\rho>0$, $t\in]0,1[$ and $x\in X$, there exists a Uryshon function $\varphi\in {\rm Lip}(X)$ for the pair $(X\setminus Q_\rho(x)), \overline{Q}_{t\rho}(x))$ such 
$$
\|{\rm Lip}\varphi\|_{L^\infty_\mu(X)}\leq{1\over\rho(1-t)},
$$
where for every $y\in X$,
$$
{\rm Lip}\varphi(y):=\limsup_{d(y,z)\to0}{|\varphi(y)-\varphi(z)|\over d(y,z)}.
$$
But, since $\mu$ is doubling and $X$ supports a weak $(1,p)$-Poincar\'e inequality, from Cheeger (see \cite[Theorem 6.1]{cheeger99}) we have ${\rm Lip}\varphi(y)=g_\varphi(y)$ for $\mu$-a.a. $y\in X$, where $g_{\varphi}$ is the minimal $p$-weak upper gradient for $\varphi$. Hence 
$$
\|D_\mu\varphi\|_{L^\infty_\mu(X;\RR^N)}\leq{\alpha\over\rho(1-t)}
$$ 
because $|D_\mu\varphi(y)|\leq\alpha|g_{\varphi}(y)|$ for $\mu$-a.a. $y\in X$. Consequently (A$_2$) holds. Finally, if moreover $(X,d)$ is a length space, from Colding and Minicozzi II (see \cite{colding-minicozzi98} and \cite[Proposition 6.12]{cheeger99}) we can assert that there exists $\beta>0$ such that for every $x\in X$, every $\rho>0$ and every $t\in]0,1[$, 
$$
\mu(Q_\rho(x)\setminus Q_{t\rho}(x))\leq 2^\beta(1-t)^\beta\mu(Q_\rho(x)),
$$
which implies \eqref{DoublINgAssUMpTiON}.
\end{proof}
 
\medskip 
 
As a consequence of Theorem \ref{FiNaL-Main-TheOrEM} and Proposition \ref{Prop-Cheeger}, we have 

\begin{corollary}\label{coro-Cheeger-2}
Under the hypotheses of Theorem {\rm\ref{cheeger-theorem}}, if moreover $(X,d)$ is a length space and if \eqref{p-coercivity} and \eqref{p-polynomial-growth} hold, then
$$
\overline{E}(u;A)=\int_A \mathcal{Q}_\mu L_x(\nabla_\mu u(x))d\mu(x) 
$$
for all $u\in W^{1,p}_\mu(X;\RR^m)$ and all $A\in\mathcal{O}(X)$, with $\mathcal{Q}_\mu L_x:\MM\to[0,\infty]$ given by
$$
\mathcal{Q}_\mu L_x(\xi):=\lim_{\rho\to 0}\inf\left\{\mint_{Q_\rho(x)}L_y(\xi+\nabla_\mu w(y))d\mu(y):w\in W^{1,p}_{\mu,0}(Q_\rho(x);\RR^m)\right\}.
$$
\end{corollary} 
 

\section{Auxiliary results}

\subsection{Interchange of infimum and integral}

Let $(A,d)$ be a locally compact metric space that is $\sigma$-compact, let $\mu$ be a positive Radon measure on $A$ and let $Y$ be a separable Banach space.

\subsubsection{The $\mu$-essential supremum of a set of $\mu$-measurable functions} Let $\mathcal{M}_\mu(A;Y)$ be the class of all closed-valued $\mu$-measurable multifunctions\footnote{A multifunction $\Gamma:A\dto Y$ is said to be closed-valued if $\Gamma(x)$ is closed for $\mu$-a.a. $x\in A$, and $\mu$-measurable if for every open set $U\subset A$, $\{x\in A:\Gamma(x)\cap U\not=\emptyset\}$ is $\mu$-measurable.} from $A$ to $Y$ and let $\mathcal{M}^*_\mu(A;Y):=\{\Gamma\in\mathcal{M}_\mu(A;Y):\Gamma(x)\not=\emptyset\hbox{ for $\mu$-a.a. }x\in A\}$. The following proposition is due to Valadier (see \cite[Proposition 14]{valadier71}).

\begin{proposition}\label{Valadier-prop}
Let $\mathcal{F}$ be a nonempty subclass of $\mathcal{M}^*_\mu(A;Y)$. Then, there exists $\Gamma\in\mathcal{M}^*_\mu(A;Y)$ such that{\rm:}
\begin{itemize}
\item[{\rm(i)}] for every $\Lambda\in\mathcal{F}$, $\Lambda(x)\subset\Gamma(x)$ for $\mu$-a.a. $x\in A;$ 
\item[{\rm(ii)}] if $\Gamma^\prime\in\mathcal{M}_\mu(A;Y)$ and if for every $\Lambda\in\mathcal{F}$, $\Lambda(x)\subset\Gamma^\prime(x)$ for $\mu$-a.a. $x\in A$, then $\Gamma(x)\subset\Gamma^\prime(x)$ for $\mu$-a.a. $x\in A$.
\end{itemize}
\end{proposition}
Note that $\Gamma$ given by Proposition \ref{Valadier-prop} is unique with respect to the equality $\mu$-a.e. Valadier called it the $\mu$-essential upper bound of $\mathcal{F}$. Here is the definition of the $\mu$-essential supremum of a set of $\mu$-measurable functions.

\begin{definition}
Let $\mathcal{H}$ be a set of $\mu$-measurable functions from $A$ to $Y$. By the $\mu$-essential supremum of $\mathcal{H}$ we mean the $\mu$-essential upper bound of $\{\{w\}:w\in\mathcal{H}\}$, where $\{w\}:A\dto Y$ is defned by $\{w\}(x):=\{w(x)\}$. Thus, if we denote the $\mu$-essential supremum of $\mathcal{H}$ by $\Gamma$, we have:
\begin{itemize}
\item[(i)] $\{w(x):w\in\mathcal{H}\}\subset\Gamma(x)$ for $\mu$-a.a $x\in A$;
\item[(ii)] if $\Gamma^\prime\in\mathcal{M}_\mu(A;Y)$ and if $\{w(x):w\in\mathcal{H}\}\subset\Gamma^\prime(x)$ for $\mu$-a.a $x\in A$, then $\Gamma(x)\subset\Gamma^\prime(x)$ for $\mu$-a.a. $x\in A$.
\end{itemize}
\end{definition}

The following lemma gives a (classical) representation of the $\mu$-essential supremum (see \cite{bouchitte-valadier88}).
\begin{lemma}\label{representation-mu-ess-sup}
Let $p\geq 1$ be a real number, let $\mathcal{H}\subset L^p_\mu(A;Y)$ and let $\Gamma$ be the $\mu$-essential supremum of $\mathcal{H}$. Then, there exists a countable subset $\mathcal{D}$ of $\mathcal{H}$ such that $\Gamma(x)={\rm cl}\{w(x):w\in\mathcal{D}\}$ for $\mu$-a.a. $x\in A$, where ${\rm cl}$ denotes the closure in $Y$.
\end{lemma}

\subsubsection{Interchange theorem}  In what follows, by a Urysohn function for a pair $(F,G)$ of disjoint closed subsets $F$ and $G$ of $A$ we mean a continuous function $\phi:A\to\RR$ such that $\phi(x)\in[0,1]$ for all $x\in A$, $\phi(x)=0$ for all $x\in F$ and $\phi(x)=1$ for all $x\in G$. Let $p\geq 1$ be a real number and let $\mathcal{H}\subset L^p_\mu(A;Y)$. The following definition was introduced in \cite{oah-jpm03}.
\begin{definition}
We say that $\mathcal{H}$ is normally decomposable if for every $w,\hat w\in\mathcal{H}$ and every $K,V\subset A$ with $K$ compact, $V$ open and $K\subset V$, there exists a Urysohn $\phi$ function for the pair $(A\setminus V,K)$ such that $\phi w+(1-\phi)\hat w\in\mathcal{H}$.
\end{definition}

Let $\{L_x\}$ be a field of Carath\'eodory integrand over $A$, i.e., to $\mu$-a.e. $x\in A$ there corresponds a continuous function $L_x:Y\to[0,\infty]$ so that the function $x\mapsto L_x(\xi)$ is $\mu$-measurable for all $\xi\in Y$. The following theorem is a consequence of \cite[Theorem 1.1]{oah-jpm03}.

\begin{theorem}\label{Interchange-Theorem}
If $\mathcal{H}$ is normally decomposable and if $\{L_x\}$ is of $p$-polynomial growth, i.e., there exists $c>0$ such that $L_x(\xi)\leq c(1+|\xi|^p)$ for all $\xi\in Y$ and all $\mu$-a.a. $x\in A$, then
$$
\inf_{w\in\mathcal{H}}\int_A L_x(w(x))d\mu(x)=\int_A\inf_{\xi\in\Gamma(x)}L_x(\xi)d\mu(x)
$$
with $\Gamma:A\dto Y$ given by the  $\mu$-essential supremum of $\mathcal{H}$.
\end{theorem}

\subsection{De Giorgi-Letta's lemma}

Let $X=(X,d)$ be a metric space, let $\mathcal{O}(X)$ be the class of all open subsets of $X$ and let $\mathcal{B}(X)$ be the class of all Borel subsets of $X$, i.e., the smallest $\sigma$-algebra containing the open (or equivalently the closed) subsets of $X$. The following result is due to De Giorgi and Letta (see \cite{degiorgi-letta77} and also \cite[Lemma 3.3.6 p. 105]{buttazzo89}).

\begin{lemma}\label{DeGiorgi-Letta-Lemma}
Let $\mathcal{S}:\mathcal{O}(X)\to[0,\infty]$ be an increasing set function, i.e., $\mathcal{S}(A)\leq \mathcal{S}(B)$ for all $A,B\in\mathcal{O}(X)$ such $A\subset B$, satisfying the following three conditions{\rm:}
\begin{itemize}
\item[{\rm(i)}] $\mathcal{S}(\emptyset)=0;$
\item[{\rm(ii)}] $\mathcal{S}$ is superadditive, i.e., $\mathcal{S}(A\cup B)\geq \mathcal{S}(A)+\mathcal{S}(B)$ for all $A,B\in\mathcal{O}(X)$ such that $A\cap B=\emptyset;$
\item[{\rm(iii)}] $\mathcal{S}$ is subadditive, i.e., $\mathcal{S}(A\cup B)\leq \mathcal{S}(A)+\mathcal{S}(B)$ for all $A,B\in\mathcal{O}(X);$
\item[{\rm(iv)}] there exists a finite Radon measure $\nu$ on $X$ such that $\mathcal{S}(A)\leq\nu(A)$ for all $A\in\mathcal{O}(X)$.
\end{itemize}
Then, $\mathcal{S}$ can be uniquely extended to a finite positive Radon measure on $X$ which is absolutely continuous with respect to $\nu$.
\end{lemma}


\section{Proof of the main results}

\subsection{Proof of Proposition \ref{MR1}} We divide the proof into four steps.

\subsection*{Step 1. Another formula for \boldmath$\overline{E}$\unboldmath} Let $\mathcal{E}:W^ {1,p}_\mu(X;\RR^m)\times\mathcal{O}(X)\to[0,\infty]$ be defined by
$$
\mathcal{E}(u;A):=\inf\Big\{E(v;A):v\in\mathcal{A}^m_u\Big\}.
$$
The following lemma makes clear the link between $\mathcal{E}$ and $\overline{E}$.
\begin{lemma}\label{Lemma1-Prop}
For every $u\in W^{1,p}_\mu(X;\RR^m)$ and every $A\in\mathcal{O}(X)$,
\begin{equation}\label{Step-1-Eq1-Prop}
\overline{E}(u;A)=\inf\left\{\liminf_{n\to\infty}\mathcal{E}(u_n;A):\mathcal{A}(X;\RR^m)\ni u_n\to u\hbox{ in }L^p_\mu(X;\RR^m)\right\}.
\end{equation}
\end{lemma}
\begin{proof}[\bf Proof of Lemma \ref{Lemma1-Prop}]
Fix $u\in W^{1,p}_\mu(X;\RR^m)$ and $A\in\mathcal{O}(X)$ and denote the right-hand side of \eqref{Step-1-Eq1-Prop} by $\overline{\mathcal{E}}(u;A)$. As $E(v;A)\geq \mathcal{E}(v;A)$ for all $v\in \mathcal{A}(X;\RR^m)$ we have $\overline{E}(u;A)\geq\overline{\mathcal{E}}(u;A)$. Thus, it remains to prove that
\begin{equation}\label{Proof-LemmA1-Eq}
\overline{\mathcal{E}}(u;A)\geq \overline{E}(u;A).
\end{equation}
Fix any $\eps>0$ and consider $\{u_n\}_n\subset \mathcal{A}(X;\RR^m)$ with $u_n\to u$ in $L^p_\mu(X;\RR^m)$ such that
$
\overline{\mathcal{E}}(u;A)+{\eps\over 2}\geq\liminf_{n\to\infty}\mathcal{E}(u_n;A).
$
To every $n\geq 1$, there corresponds $v_n\in\mathcal{A}^m_u$ such that 
$
\mathcal{E}(u_n;A)+{\eps\over 2}\geq E(v_n;A).
$
Hence $v_n\to u$ in $L^p_\mu(X;\RR^m)$ and
$
\overline{\mathcal{E}}(u;A)+\eps\geq\liminf_{n\to\infty}E(v_n;A)\geq\overline{E}(u;A),
$
and \eqref{Proof-LemmA1-Eq} follows by letting $\eps\to0$.
\end{proof}
 
\subsection*{Step 2. Integral representation of the variational functional \boldmath$\mathcal{E}$\unboldmath} We now establish an integral representation for the variational function $\mathcal{E}$. First of all, it easy to see that
$$
\mathcal{E}(u;A)=\inf_{w\in\mathcal{H}^m_u(A)}\int_AL_x(w(x))d\mu(x)
$$
for all $u\in \mathcal{A}(X;\RR^m)$ and all $A\in\mathcal{O}(X)$, where $\mathcal{H}^m_u(A)$ is given by
$$
\mathcal{H}^m_u(A):=\Big\{w\in L^\infty_\mu(A;\MM):w(x)=\nabla v(x)\hbox{ for $\mu$-a.a. }x\in A\hbox{ with }v\in\mathcal{A}^m_u\Big\}
$$
with $\mathcal{A}^m_u$ defined in \eqref{Def-of-Amu}. On the other hand, we have
\begin{lemma}\label{Lemma2-Prop}
The set $\mathcal{H}^m_u(A)$ is normally decomposable for all $u\in W^{1,p}_\mu(X;\RR^m)$ and all $A\in\mathcal{O}(X)$.
\end{lemma}
\begin{proof}[\bf Proof of Lemma \ref{Lemma2-Prop}]
Let $u\in\mathcal{A}(X;\RR^m)$ and $A\in\mathcal{O}(X)$. Fix $K,V\subset A$ with $K$ compact, $V$ open and $K\subset V$, fix $w,\hat w\in\mathcal{H}^m_u(A)$ and consider $v,\hat v\in\mathcal{A}^m_u$ such that for $\mu$-a.e. $x\in A$, $w(x)=\nabla v(x)$ and $\hat w(x)=\nabla \hat v(x)$. As $\mathcal{A}(X)$ satisfies the Uryshon property, there exists a Uryshon function $\varphi\in \mathcal{A}(X)$ for the pair $(X\setminus V,K)$. Then $\phi:=\varphi|_{A}$ is a Uryshon function for the pair $(A\setminus V,K)$. On the other hand, using \eqref{Gradient-Property} we have $\nabla(\varphi v+(1-\varphi)\hat v)=\varphi\nabla v+(1-\varphi)\nabla\hat v+D\varphi\otimes(v-\hat v)$, and so $\nabla(\varphi v+(1-\varphi)\hat v)(x)=\phi(x) w(x)+(1-\phi(x))\hat w(x)$ for $\mu$-a.a. $x\in A$. Moreover, $\varphi v+(1-\varphi)\hat v\in\mathcal{A}^m_u$ and consequently $\phi w+(1-\phi)\hat w\in \mathcal{H}^m_u(A)$.
\end{proof}

\medskip

From now on, fix $u\in \mathcal{A}(X;\RR^m)$ and $A\in\mathcal{O}(X)$. As $\{L_x\}$ is of $p$-polynomial growth, i.e., \eqref{p-polynomial-growth} holds, and $\mathcal{H}^m_u(A)$ is normally decomposable by Lemma \ref{Lemma2-Prop}, from Theorem \ref{Interchange-Theorem} we deduce that
\begin{equation}\label{FirsT-IntEGRAL-ReprESentation}
\mathcal{E}(u;A)=\int_A\inf_{\xi\in\Gamma_u(x,A)}L_x(\xi)d\mu(x)
\end{equation}
with $\Gamma(\cdot,A):A\dto\MM$ given by the $\mu$-essential supremum of $\mathcal{H}_u^m(A)$.

\subsection*{Step 3. Refining the integral representation of \boldmath$\mathcal{E}$\unboldmath} Finally, to refine the integral representation of $\mathcal{E}$ in \eqref{FirsT-IntEGRAL-ReprESentation}, we need the following lemma.

\begin{lemma}
$\Gamma_u(x,A)=N^m_\mu(x)+\{\nabla_\mu u(x)\}$ for $\mu$-a.a. $x\in A$.
\end{lemma}
\begin{proof}
From Lemma \ref{representation-mu-ess-sup} there exists a countable subset $\mathcal{D}_u(A)$ of $\mathcal{H}_u^m(A)$ such that 
\begin{equation}\label{AppLIcation-mu-ess-sup-lemma}
\Gamma_u(x,A)={\rm cl}\big\{w(x):w\in\mathcal{D}_u(A)\big\}\hbox{ for $\mu$-a.a. $x\in A$.}
\end{equation}
Fix any $w\in\mathcal{D}_u(A)$. By definition of $\mathcal{H}^m_u(A)$, there exists $v\in\mathcal{A}^m_u$ such that for $\mu$-a.e. $x\in A$, $w(x)=\nabla v(x)$. But $\nabla_\mu v(x)=\nabla_\mu u(x)$ for $\mu$-a.a. $x\in X$ by Lemma \ref{comp-EquA-mu-aa}, hence $w(x)=\nabla v(x)-\nabla_\mu v(x)+\nabla_\mu u(x)$ for $\mu$-a.a. $x\in A$, where for $\mu$-a.e. $x\in A$, $\nabla v(x)-\nabla_\mu u(x)\in N^m_\mu(x)$, and so $w(x)\in N_\mu^m(x)+\{\nabla_\mu u(x)\}$ for $\mu$-a.a. $x\in A$. Thus $\{w(x):w\in\mathcal{D}_u(A)\}\subset N^m_\mu(x)+\{\nabla_\mu u(x)\}$ for $\mu$-a.a. $x\in A$. Using \eqref{AppLIcation-mu-ess-sup-lemma} it follows that 
$$
\Gamma_u(x,A)\subset N^m_\mu(x)+\{\nabla_\mu u(x)\}\hbox{ for $\mu$-a.a. $x\in A$.}
$$
Let $\Gamma_u:X\dto\MM$ be the $\mu$-essential supremum of $\mathcal{H}^m_u$ (which corresponds to $\mathcal{H}^m_u(A)$ with $A=X$). If $w\in\mathcal{H}^m_u$ then $w|_A\in\mathcal{H}^m_u(A)$, and so $\Gamma_u(x)\subset\Gamma_u(x,A)$ for $\mu$-a.a. $x\in A$ because, from Lemma \ref{representation-mu-ess-sup}, $\Gamma_u(x)={\rm cl}\{w(x):w\in\mathcal{D}_u\}$ for $\mu$-a.a. $x\in A$ with $\mathcal{D}_u$ a countable subset of $\mathcal{H}^m_u$. Hence, the proof is completed by showing that
\begin{equation}\label{GoAL-FirSt-PrOOf}
N^m_\mu(x)+\{\nabla_\mu u(x)\}\subset \Gamma_u(x)\hbox{ for $\mu$-a.a. $x\in A$.}
\end{equation}
For $\mu$-a.e. $x\in A$, let $\xi\in N^m_\mu(x)+\{\nabla_\mu u(x)\}$. Then $P_\mu(x)(\xi-\nabla u(x))=0$, hence $\xi-\nabla u(x)\in N^m_\mu(x)$ and so there exists $\hat w\in\mathcal{H}^m_0$ (which corresponds to $\mathcal{H}^m_u$ with $u=0$) such that  $\xi-\nabla u(x)=\hat w(x)$. Setting $\bar w:=\nabla u$ we then have $\xi=(\bar w+\hat w)(x)$ with $\bar w+\hat w\in\mathcal{H}^m_u$. Thus $N^m_\mu(x)+\{\nabla_\mu u(x)\}\subset\{w(x):w\in\mathcal{H}^m_u\}$ for $\mu$-a.a. $x\in A$, and \eqref{GoAL-FirSt-PrOOf} follows because, by definition of the $\mu$-essential supremum, $\{w(x):w\in\mathcal{H}^m_u\}\subset\Gamma_u(x)$ for $\mu$-a.a. $x\in X$.
\end{proof}

\medskip

This completes the proof of Proposition \ref{MR1}. $\blacksquare$

\subsection{Proof of Theorem \ref{Main-Theorem-Convex}} Fix $A\in\mathcal{O}(X)$ and define the increasing set $\mathcal{F}:W^{1,p}_\mu(X;\RR^m)\to[0,\infty]$ by
\begin{equation}\label{NeW-FunCtiONAL}
\mathcal{F}(u):=\int_A\widehat{L}_x(\nabla_\mu u(x))d\mu(x).
\end{equation}
(Note that $\mathcal{F}(u)=\widehat{\mathcal{E}}(u;A)$ for all $u\in\mathcal{A}(X;\RR^m)$, where $\widehat{\mathcal{E}}(\cdot;A):W^{1,p}_\mu(X;\RR^m)\to[0,\infty]$ is defined in \eqref{DeF-fUnCt-RemarK-Local}.) From Remark \ref{p-coercivity-p-polynomial-growth-remark} we see that $\{\widehat{L}_x\}$ is both $p$-coercive and of $p$-polynomial growth, i.e.,
\begin{equation}\label{EqUaT-ThEoReM-SecoND-ResuLT1}
C|\xi|^p\leq\widehat{L}_x(\xi)\leq c(1+|\xi|^p)\hbox{ for all }\xi\in T^m_\mu(x)\hbox{ and $\mu$-a.a. }x\in X
\end{equation}
with $C>0$ and $c>0$ given respectively by \eqref{p-coercivity} and \eqref{p-polynomial-growth}. Recalling that $\mu$ is finite and using the second inequality in \eqref{EqUaT-ThEoReM-SecoND-ResuLT1} and the continuity of $L_x$ (see Remark \ref{continuity-remaRk}), from Vitali's convergence theorem we deduce that $\mathcal{F}$ is continuous with respect to the strong convergence in $W^{1,p}_\mu(X;\RR^m)$. Hence, recalling that $\mathcal{A}(X;\RR^m)$ is dense in $W^{1,p}_\mu(X;\RR^m)$ with respect to the strong convergence in $W^{1,p}_\mu(X;\RR^m)$ and taking Proposition \ref{MR1} into account, for all $u\in W^{1,p}_\mu(X;\RR^m)$ there is $\{u_n\}_n\subset \mathcal{A}(X;\RR^m)$ such that:
\begin{itemize}
\item[$\bullet$] $u_n\to u$ in $W^{1,p}_\mu(X;\RR^m)$ and so $u_n\to u$ in $L^p_\mu(X;\RR^m)$;
\item[$\bullet$] $\mathcal{F}(u)=\lim_{n\to\infty}\mathcal{F}(u_n)=\lim_{n\to\infty}\widehat{\mathcal{E}}(u_n;A)\geq\overline{E}(u;A)$,
\end{itemize}
which shows that $\mathcal{F}\geq \overline{E}(\cdot;A)$. Define $\overline{\mathcal{F}}^w,\overline{\mathcal{F}}^s:W^{1,p}_\mu(X;\RR^m)\to[0,\infty]$ by:
\begin{itemize}
\item[$\bullet$] $\displaystyle\overline{\mathcal{F}}^w(u):=\inf\left\{\liminf_{n\to\infty}\mathcal{F}(u_n):u_n\wto u\hbox{ in }W^{1,p}_\mu(X;\RR^m)\right\}$;
\item[$\bullet$] $\displaystyle\overline{\mathcal{F}}^s(u):=\inf\left\{\liminf_{n\to\infty}\mathcal{F}(u_n):u_n\to u\hbox{ in }L^p_\mu(X;\RR^m)\right\}$.
\end{itemize}
Since $\widehat{L}_x$ is convex for $\mu$-a.a. $x\in X$, the functional $\mathcal{F}$ is convex, and so $\overline{\mathcal{F}}^w=\mathcal{F}$ because $\mathcal{F}$ is strongly continuous in $W^{1,p}_\mu(X;\RR^m)$. On the other hand, consider $u\in W^{1,p}_\mu(X;\RR^m)$ and $\{u_n\}_n\subset W^{1,p}_\mu(X;\RR^m)$ such that $u_n\to u$ in $L^p_\mu(X;\RR^m)$ and $\liminf_{n\to\infty}\mathcal{F}(u_n)=\lim_{n\to\infty}\mathcal{F}(u_n)<\infty$. Using the first inequality in \eqref{EqUaT-ThEoReM-SecoND-ResuLT1} we deduce that $\sup_n\|u_n\|_{W^{1,p}_\mu(X;\RR^m)}<\infty$, hence (up to a subsequence) $u_n\wto u$ in $W^{1,p}_\mu(X;\RR^m)$ because $W^{1,p}_\mu(X;\RR^m)$ is reflexive (since $p\in]1,\infty[$, see Remark \ref{ReFLeXIVITY-of-Metric-SObOlev-Spaces}), and so $\liminf_{n\to\infty}\mathcal{F}(u_n)\geq\overline{\mathcal{F}}^w(u)$. Thus $\overline{\mathcal{F}}^s\geq\overline{\mathcal{F}}^w$ and consequently $\overline{\mathcal{F}}^s=\mathcal{F}$ because $\mathcal{F}\geq \overline{\mathcal{F}}^s$. As $\widehat{\mathcal{E}}(\cdot;A)\geq\mathcal{F}$ we have $\overline{E}(\cdot;A)\geq\overline{\mathcal{F}}^s$ by using Proposition \ref{MR1}, and the proof is complete. $\blacksquare$

\begin{remark}
From the proof of Theorem \ref{Main-Theorem-Convex} we can extract the following lemma which asserts that for $A\in\mathcal{O}(X)$ and when \eqref{p-polynomial-growth} is satisfied, $\overline{E}(\cdot,A)$ is the lower semicontinuous envelope of $\mathcal{F}$ defined in \eqref{NeW-FunCtiONAL} with respect to the strong convergence in $L^p_\mu(X;\RR^m)$.
\begin{lemma}\label{LeMMa-Main-Convex-Theorem}
If \eqref{p-polynomial-growth} holds then 
\begin{equation}\label{NeW-FoRmUlA}
\overline{E}(u;A)=\inf\left\{\liminf_{n\to\infty}\int_A\widehat{L}_x(\nabla_\mu u_n(x))d\mu(x):W^{1,p}_\mu(X;\RR^m)\ni u_n\stackrel{L^p_\mu}{\to} u\right\}
\end{equation}
for all $u\in W^{1,p}_\mu(X;\RR^m)$ and all $A\in\mathcal{O}(X)$.
\end{lemma}
\end{remark}

\subsection{Proof of Theorem \ref{N-C-IRT}} Fix $u\in W^{1,p}_\mu(X;\RR^m)$ and define $\mathcal{S}_u:\mathcal{O}(X)\to[0,\infty]$ by 
$$
\mathcal{S}_u(A):=\overline{E}(u;A).
$$ 
Taking Lemma \ref{LeMMa-Main-Convex-Theorem} into account and using the second inequality in \eqref{EqUaT-ThEoReM-SecoND-ResuLT1} we see that 
\begin{equation}\label{ImPorTanT-EqUAtION}
\mathcal{S}_u(A)\leq\int_A c(1+|\nabla_\mu u(x)|^p)d\mu(x)\hbox{ for all }A\in\mathcal{O}(X).
\end{equation}
Thus, the condition (iv) of Lemma \ref{DeGiorgi-Letta-Lemma} is satisfied with $\nu=c(1+|\nabla_\mu u|^p)d\mu$ (which is absolutely continuous with respect to $\mu$). On the other hand, it is easily seen that the conditions (i) and (ii) of Lemma \ref{DeGiorgi-Letta-Lemma} are satisfied. Hence, the proof is completed by proving  the condition (iii) of Lemma \ref{DeGiorgi-Letta-Lemma}, i.e., 
\begin{equation}\label{Subadditivity-First-Goal}
\mathcal{S}_u(A\cup B)\leq \mathcal{S}_u(A)+\mathcal{S}_u(B)\hbox{ for all }A,B\in\mathcal{O}(X). 
\end{equation}
Indeed, by Lemma \ref{DeGiorgi-Letta-Lemma}, the set function $\mathcal{S}_u$ can be (uniquely) extended to a (finite) positive Radon measure which is absolutely continuous with respect to $\mu$, and the theorem follows by using Radon-Nikodym's theorem and then Lebesgue's differentiation theorem. To show \eqref{Subadditivity-First-Goal} we need the following lemma.
\begin{lemma}\label{LeMMa-MaiN-TheOReM1}
If $U,V,Z,T\in\mathcal{O}(X)$ are such that $\overline{Z}\subset U$ and $T\subset V$,  then
\begin{equation}\label{SubAddiTive-Goal}
\mathcal{S}_u(Z\cup T)\leq\mathcal{S}_u(U)+\mathcal{S}_u(V).
\end{equation}
\end{lemma}
\begin{proof}[\bf Proof of Lemma \ref{LeMMa-MaiN-TheOReM1}]
Let $\{u_n\}_n$ and $\{v_n\}_n$ be two sequences in $\mathcal{A}(X;\RR^m)$ such that:
\begin{eqnarray}
&& u_n\to u\hbox{ in }L^p_\mu(X;\RR^m);\label{PrOoF-MT2-EquA1}\\
&& v_n\to u\hbox{ in }L^p_\mu(X;\RR^m);\label{PrOoF-MT2-EquA2}\\
&& \lim_{n\to\infty}\int_UL_x(\nabla u_n(x))d\mu(x)=\mathcal{S}_u(U)<\infty;\label{PrOoF-MT2-EquA3}\\
&& \lim_{n\to\infty}\int_VL_x(\nabla v_n(x))d\mu(x)=\mathcal{S}_u(V)<\infty.\label{PrOoF-MT2-EquA4}
\end{eqnarray}
Fix $\delta\in]0,{\rm dist}(Z,\partial U)[$ with $\partial U:=\overline{U}\setminus U$, fix any $n\geq 1$ and any $q\geq 1$ and consider $W^-_i,W^+_i\subset X$ given by:
\begin{itemize} 
\item[$\bullet$]$W^-_i:=Z^-\left({\delta\over 3}+{(i-1)\delta\over 3q}\right)=\left\{x\in X:{\rm dist}(x,Z)\leq {\delta\over 3}+{(i-1)\delta\over 3q}\right\}$;
\item[$\bullet$]$W^+_i:=Z^+\left({\delta\over 3}+{i\delta\over 3q}\right)=\left\{x\in X:{\delta\over 3}+{i\delta\over 3q}\leq{\rm dist}(x,Z)\right\},$
\end{itemize}
where $i\in\{1,\cdots,q\}$. As $\mathcal{A}(X)$ satisfies the Uryshon property, for every $i\in\{1,\cdots,q\}$ there exists a Uryshon function $\varphi_i\in\mathcal{A}(X)$ for the pair $(W^+_i,W^-_i)$. Define $w_n^i\in\mathcal{A}(X;\RR^m)$ by 
$$
w^i_n:=\varphi_iu_n+(1-\varphi_i)v_n.
$$ 
Setting $W_i:=X\setminus (W^-_i\cup W^{+}_i)$ and using \eqref{Gradient-Property} and \eqref{Der-Prop-1} we have
$$
\nabla w_n^i=\left\{
\begin{array}{ll}
\nabla u_n&\hbox{in }W^-_i\\
D\varphi_i\otimes(u_n-v_n)+\varphi_i\nabla u_n+(1-\varphi_i)\nabla v_n&\hbox{in }W_i\\
\nabla v_n&\hbox{in }W^+_i.
\end{array}
\right.
$$
Noticing that $Z\cup T=((Z\cup T)\cap W^-_i)\cup(W\cap W_i)\cup(T\cap W^+_i)$ with $(Z\cup T)\cap W_i^-\subset U$, $T\cap W^+_i\subset V$ and $W:=T\cap\{x\in U:{\delta\over 3}<{\rm dist}(x,Z)<{2\delta\over 3}\}$ we deduce that 
\begin{eqnarray}
\int_{Z\cup T}L_x(\nabla w^i_n)d\mu&\leq&\int_UL_x(\nabla u_n)d\mu+\int_VL_x(\nabla v_n)d\mu\label{LayERs-Eq1}\\
&&+\int_{W\cap W_i}L_x(\nabla w^i_n)d\mu\nonumber
\end{eqnarray}
for all $i\in\{1,\cdots,q\}$. Moreover, from \eqref{p-polynomial-growth} we see that for each $i\in\{1,\cdots,q\}$,
\begin{eqnarray}
\int_{W\cap W_i}L_x(\nabla w^i_n)d\mu&\leq&\alpha\|D\varphi_i\|^p_{L^\infty_\mu(X;\RR^N)}\|u_n-v_n\|^p_{L^p_\mu(X;\RR^m)}\label{LayERs-Eq2}\\
&&+\alpha\int_{W\cap W_i}(1+|\nabla u_n|^p+|\nabla v_n|^p)d\mu\nonumber
\end{eqnarray}
with $\alpha:=2^{2p}c$. Substituting \eqref{LayERs-Eq2} into \eqref{LayERs-Eq1} and averaging these inequalities, it follows that for every $n\geq 1$ and every $q\geq 1$, there exists $i_{n,q}\in\{1,\cdots,q\}$ such that
\begin{eqnarray}
\int_{Z\cup T}L_x(\nabla w_n^{i_{n,q}})d\mu&\leq&\int_UL_x(\nabla u_n)d\mu+\int_VL_x(\nabla v_n)d\mu\nonumber\\
&&+{\alpha\over q}\sum_{i=1}^q\|D\varphi_i\|^p_{L^\infty_\mu(X;\RR^N)}\|u_n-v_n\|^p_{L^p_\mu(X;\RR^m)}\nonumber\\
&&+{\alpha\over q}\left(\mu(X)+\int_U|\nabla u_n|^pd\mu+\int_V|\nabla v_n|^pd\mu\right).\nonumber
\end{eqnarray}
On the other hand, by \eqref{PrOoF-MT2-EquA1} and \eqref{PrOoF-MT2-EquA2} we have:
\begin{itemize}
\item[$\bullet$] $\displaystyle\lim_{n\to\infty}\|u_n-v_n\|^p_{L^p_\mu(X;\RR^m)}=0$;
\item[$\bullet$] $\displaystyle\lim_{n\to\infty}\|w_n^{i_{n,q}}-u\|^p_{L^p_\mu(X;\RR^m)}=0$ for all $q\geq 1$.
\end{itemize}
Moreover, using \eqref{PrOoF-MT2-EquA3} and \eqref{PrOoF-MT2-EquA4} together with \eqref{p-coercivity} we see that:
\begin{itemize}
\item[$\bullet$] $\displaystyle\limsup_{n\to\infty}\int_U|\nabla u_n(x)|^pd\mu(x)<\infty$;
\item[$\bullet$] $\displaystyle\limsup_{n\to\infty}\int_V|\nabla v_n(x)|^pd\mu(x)<\infty$.
\end{itemize}
Letting $n\to\infty$ (and taking \eqref{PrOoF-MT2-EquA3} and \eqref{PrOoF-MT2-EquA4} into account) we deduce that for every $q\geq 1$,
\begin{equation}\label{Limit-AvErAgE-LaYerS}
\mathcal{S}_u(Z\cup T)\leq\liminf_{n\to\infty}\int_{Z\cup T}L_x(\nabla w_n^{i_{n,q}}(x))d\mu(x)\leq\mathcal{S}_u(U)+\mathcal{S}_u(V)+{\hat\alpha\over q}
\end{equation}
with $\hat\alpha:=\alpha(\mu(X)+\limsup_{n\to\infty}\int_U|\nabla u_n(x)|^pd\mu(x)+\limsup_{n\to\infty}\int_V|\nabla v_n(x)|^pd\mu(x))$, and \eqref{SubAddiTive-Goal} follows from \eqref{Limit-AvErAgE-LaYerS} by letting $q\to\infty$. 
\end{proof}

\medskip

We now prove \eqref{Subadditivity-First-Goal}. Fix $A,B\in\mathcal{O}(X)$. Fix any $\eps>0$ and consider $C,D\in\mathcal{O}(X)$ such that $\overline{C}\subset A$, $\overline{D}\subset B$ and 
$$
\int_Ec(1+|\nabla_\mu u(x)|^p)d\mu(x)<\eps
$$ 
with $E:=A\cup B\setminus\overline{C\cup D}$. Then $\mathcal{S}_u(E)\leq\eps$ by \eqref{ImPorTanT-EqUAtION}. Let $\hat C,\hat D\in\mathcal{O}(X)$ be such that $\overline{C}\subset\hat C$, $\overline{\hat C}\subset A$, $\overline{D}\subset\hat D$ and $\overline{\hat D}\subset B$. Applying Lemma \ref{LeMMa-MaiN-TheOReM1} with $U=\hat C\cup\hat D$, $V=T=E$ and $Z=C\cup D$ (resp. $U=A$, $V=B$, $Z=\hat C$ and $T=\hat D$) we obtain
$$
\mathcal{S}_u(A\cup B)\leq\mathcal{S}_u(\hat C\cup\hat D)+\eps\hbox{ \big(resp. }\mathcal{S}_u(\hat C\cup\hat D)\leq\mathcal{S}_u(A)+\mathcal{S}_u(B)\big),
$$
and \eqref{Subadditivity-First-Goal} follows by letting $\eps\to0$. $\blacksquare$

\begin{remark}
The method used in the proof of Lemma \ref{LeMMa-MaiN-TheOReM1} is a variant of the so-called De Giorgi's slicing method. Note that the proof of Lemma \ref{LeMMa-MaiN-TheOReM1} can be rewritten (exactly in the same way) in considering \eqref{NeW-FoRmUlA} together with \eqref{EqUaT-ThEoReM-SecoND-ResuLT1}, in using \eqref{MU-Gradient-Property} and \eqref{MU-Der-Prop-1} instead of \eqref{Gradient-Property} and \eqref{Der-Prop-1}, and in replacing (in the text of the proof) the integrand ``$L_x$" by ``$\widehat{L}_x$", the space ``$\mathcal{A}(X;\RR^m)$" by ``$W^{1,p}_\mu(X;\RR^m)$", the operator ``$D$" by ``$D_\mu$" and the operator ``$\nabla$" by ``$\nabla_\mu$". On the other hand, by De Giorgi's slicing method we can also establish the following result.
\begin{lemma}\label{AddiTiONaL-LEMma}
If \eqref{EqUaT-ThEoReM-SecoND-ResuLT1} holds (which is the case when \eqref{p-coercivity} and \eqref{p-polynomial-growth} are satisfied) then 
\begin{equation}\label{NeW-FoRmUlA-Bis}
\overline{E}(u;A)=\inf\left\{\liminf_{n\to\infty}\int_A\widehat{L}_x(\nabla_\mu u_n(x))d\mu(x):W^{1,p}_{\mu,0}(A;\RR^m)\ni u_n-u\stackrel{L^p_\mu}{\to} 0\right\}
\end{equation}
for all $u\in W^{1,p}_\mu(X;\RR^m)$. 
\end{lemma}
\begin{proof}
Fix $u\in W^{1,p}_\mu(X;\RR^m)$ and $A\in\mathcal{O}(X)$ and denote the right-hand side of \eqref{NeW-FoRmUlA-Bis} by $\overline{\mathcal{E}}(u;A)$. Taking Lemma \ref{LeMMa-Main-Convex-Theorem} into account and noticing that $W^{1,p}_{\mu,0}(A;\RR^m)\subset W^{1,p}_\mu(X;\RR^m)$ we have $\overline{\mathcal{E}}(u;A)\geq\overline{E}(u;A)$. Thus, it remains to prove that
\begin{equation}\label{GOal-Lemma-Bis}
\overline{\mathcal{E}}(u;A)\leq\overline{E}(u;A).
\end{equation}
Let $\{u_n\}_n\subset W^{1,p}_\mu(X;\RR^m)$ be such that
\begin{eqnarray}
&& u_n\to u\hbox{ in }L^p_\mu(X;\RR^m)\label{PrOoF-MT2-EquA1-BiS};\\
&& \lim_{n\to\infty}\int_A\widehat{L}_x(\nabla_\mu u_n(x))d\mu(x)=\overline{E}(u;A)<\infty.\label{PrOoF-MT2-EquA3-BiS}
\end{eqnarray}
Fix $\delta>0$ and set $A_\delta:=\{x\in A:{\rm dist}(x,\partial A)>\delta\}$ with $\partial A:=\overline{A}\setminus A$. Fix any $n\geq 1$ and any $q\geq 1$ and consider $W^-_i,W^+_i\subset X$ given by
\begin{itemize} 
\item[$\bullet$]$W^-_i:=A^-_\delta\left({\delta\over 3}+{(i-1)\delta\over 3q}\right)=\left\{x\in X:{\rm dist}(x,A_\delta)\leq {\delta\over 3}+{(i-1)\delta\over 3q}\right\}$;
\item[$\bullet$]$W^+_i:=A^+_\delta\left({\delta\over 3}+{i\delta\over 3q}\right)=\left\{x\in X:{\delta\over 3}+{i\delta\over 3q}\leq{\rm dist}(x,A_\delta)\right\}$,
\end{itemize}
where $i\in\{1,\cdots,q\}$. (Note that $W^-_i\subset A$.) As $\mathcal{A}(X)$ satisfies the Uryshon property, for every $i\in\{1,\cdots,q\}$ there exists a Uryshon function $\varphi_i\in\mathcal{A}(X)$ for the pair $(W^+_i,W^-_i)$. Define $w_n^i:X\to\RR^m$  by 
$$
w^i_n:=\varphi_iu_n+(1-\varphi_i)u.
$$ 
Then $w_n^i-u\in W^{1,p}_{\mu,0}(A;\RR^m)$ (see Remark \ref{Remark-MU-Gradient-Property-BiS}). Setting $W_i:=X\setminus (W^-_i\cup W^{+}_i)\subset A$ and using \eqref{MU-Gradient-Property-BiS} and \eqref{MU-Der-Prop-1} we have
$$
\nabla_\mu w_n^i=\left\{
\begin{array}{ll}
\nabla_\mu u_n&\hbox{in }W^-_i\\
D_\mu\varphi_i\otimes(u_n-u)+\varphi_i\nabla_\mu u_n+(1-\varphi_i)\nabla_\mu u&\hbox{in }W_i\\
\nabla_\mu u&\hbox{in }W^+_i.
\end{array}
\right.
$$
Noticing that $A=W^-_i\cup W_i\cup(A\cap W^+_i)$ we deduce that for every $i\in\{1,\cdots,q\}$,
\begin{eqnarray}
\int_A\widehat{L}_x(\nabla_\mu w^i_n)d\mu&\leq&\int_A\widehat{L}_x(\nabla_\mu u_n)d\mu+\int_{A\cap W^+_i}\widehat{L}_x(\nabla_\mu u)d\mu\label{LayERs-Eq1-BiS}\\
&&+\int_{W_i}\widehat{L}_x(\nabla_\mu w^i_n)d\mu.\nonumber
\end{eqnarray}
Moreover, from the second inequality in \eqref{EqUaT-ThEoReM-SecoND-ResuLT1} we see that for each $i\in\{1,\cdots,q\}$,
\begin{eqnarray}
\int_{W_i}\widehat{L}_x(\nabla_\mu w^i_n)d\mu&\leq&\alpha\|D_\mu\varphi_i\|^p_{L^\infty_\mu(X;\RR^N)}\|u_n-u\|^p_{L^p_\mu(X;\RR^m)}\label{LayERs-Eq2-BiS}\\
&&+\alpha\int_{W_i}(1+|\nabla_\mu u_n|^p+|\nabla_\mu u|^p)d\mu\nonumber
\end{eqnarray}
with $\alpha:=2^{2p}c$. Substituting \eqref{LayERs-Eq2-BiS} into \eqref{LayERs-Eq1-BiS} and averaging these inequalities, it follows that for every $n\geq 1$ and every $q\geq 1$, there exists $i_{n,q}\in\{1,\cdots,q\}$ such that
\begin{eqnarray}
\int_{A}\widehat{L}_x(\nabla_\mu w_n^{i_{n,q}})d\mu&\leq&\int_A\widehat{L}_x(\nabla_\mu u_n)d\mu+{1\over q}\int_A\widehat{L}_x(\nabla_\mu u)d\mu\nonumber\\
&&+{\alpha\over q}\sum_{i=1}^q\|D_\mu\varphi_i\|^p_{L^\infty_\mu(X;\RR^N)}\|u_n-u\|^p_{L^p_\mu(X;\RR^m)}\nonumber\\
&&+{\alpha\over q}\left(\mu(A)+\int_A|\nabla_\mu u_n|^pd\mu+\int_A|\nabla_\mu u|^pd\mu\right).\nonumber
\end{eqnarray}
On the other hand, by \eqref{PrOoF-MT2-EquA1-BiS} we have
$$
\lim_{n\to\infty}\|w_n^{i_{n,q}}-u\|^p_{L^p_\mu(X;\RR^m)}=0\hbox{ for all }q\geq 1.
$$
Moreover, using \eqref{PrOoF-MT2-EquA3-BiS}  together with the first inequality in \eqref{EqUaT-ThEoReM-SecoND-ResuLT1} we see that
$$
\limsup_{n\to\infty}\int_A|\nabla_\mu u_n(x)|^pd\mu(x)<\infty.
$$
Letting $n\to\infty$ (and taking \eqref{PrOoF-MT2-EquA3-BiS} into account) we deduce that for every $q\geq 1$,
\begin{equation}\label{Limit-AvErAgE-LaYerS-BiS}
\overline{\mathcal{E}}(u;A)\leq\liminf_{n\to\infty}\int_{A}\widehat{L}_x(\nabla_\mu w_n^{i_{n,q}})d\mu\leq\overline{E}(u;A)+{1\over q}\int_A\widehat{L}_x(\nabla_\mu u)d\mu+{\hat\alpha\over q}
\end{equation}
with $\hat\alpha:=\alpha(\mu(A)+\limsup_{n\to\infty}\int_A|\nabla_\mu u_n(x)|^pd\mu(x)+\int_A|\nabla_\mu u(x)|^pd\mu(x))$, and \eqref{GOal-Lemma-Bis} follows from \eqref{Limit-AvErAgE-LaYerS-BiS} by letting $q\to\infty$. 
\end{proof}
\end{remark}

\subsection{Proof of Theorem \ref{IR-Theorem}} The proof is adapted from \cite[Lemmas 3.3 and 3.5]{boufonmas98} (see also \cite[\S 2]{bouchitte-bellieud00}). Fix $u\in W^{1,p}_\mu(X;\RR^m)$ and define the set function ${\rm m}_u:\mathcal{O}(X)\to[0,\infty]$ by
$$
{\rm m}_u(A):=\inf\left\{\int_{A}\widehat{L}_x(\nabla_\mu v(x))d\mu(x):v-u\in W^{1,p}_{\mu,0}(A;\RR^m)\right\}.
$$
For each $\eps>0$ and each $A\in\mathcal{O}(X)$, denote the class of all countable family $\{Q_i:=Q_{\rho_i}(x_i)\}_{i\in I}$ of disjoint open balls of $A$ with $x_i\in A$, $\rho_i={\rm diam}(Q_i)\in]0,\eps[$ and $\mu(\partial Q_i)=0$ such that $\mu(A\setminus\cup_{i\in I}Q_i)=0$ by $\mathcal{V}_\eps(A)$, consider ${\rm m}_u^\eps:\mathcal{O}(X)\to[0,\infty]$ given by
$$
{\rm m}_u^\eps(A):=\inf\left\{\sum_{i\in I}{\rm m}_u(Q_i):\{Q_i\}_{i\in I}\in \mathcal{V}_\eps(A)\right\}
$$
and define ${\rm m}^*_u:\mathcal{O}(X)\to[0,\infty]$ by 
$$
{\rm m}^*_u(A):=\sup_{\eps>0}{\rm m}^\eps_u(A)=\lim_{\eps\to0}{\rm m}_u^\eps(A).
$$
(Note that as $X$ satisfies the Vitali covering theorem, see (C$_2$) and Remark \ref{ReMArK-VItALi-For-OpEN-SEtS}, we have $\mathcal{V}_\eps(A)\not=\emptyset$ for all $A\in\mathcal{O}(X)$ and all $\eps>0$.) 

\subsection*{Step 1. We prove that \boldmath${\rm m}^*_u(A)=\overline{E}(u;A)$\unboldmath\ for all \boldmath$A\in\mathcal{O}(X)$\unboldmath}

Taking Lemma \ref{AddiTiONaL-LEMma} into account, it is easy to see that ${\rm m}_u(A)\leq\overline{E}(u;A)$ and so ${\rm m}^*_u(A)\leq \overline{E}(u;A)$ (because in the proof of Theorem \ref{N-C-IRT} it is established that $\overline{E}(u;\cdot)$ can be uniquely extended to a finite positive Radon measure on $X$). Hence, it remains to prove that
\begin{equation}\label{EEEEqqqq}
\overline{E}(u;A)\leq {\rm m}^*_u(A)
\end{equation}
with ${\rm m}^*_u(A)<\infty$. Fix any $\eps>0$. Given $A\in\mathcal{O}(X)$, by definition of ${\rm m}^\eps_u(A)$, there exists $\{Q_{i}\}_{i\in I}\in\mathcal{V}_\eps(A)$  such that 
\begin{equation}\label{EEEqqq1}
\sum_{i\in I}{\rm m}_u(Q_{i})\leq {\rm m}^\eps_u(A)+{\eps\over 2}.
\end{equation}
Given any $i\in I$, by definition of ${\rm m}_u(Q_i)$, there exists $v_i\in W^{1,p}_{\mu}(Q_{i};\RR^m)$ such that $v_i-u\in W^{1,p}_{\mu,0}(Q_{i};\RR^m)$ and
\begin{equation}\label{EEEqqq1-bis}
\int_{Q_i}\widehat{L}_x(\nabla_\mu v_i(x))d\mu(x)\leq {\rm m}_u(Q_i)+{\eps \mu(Q_i)\over 2\mu(A)}.
\end{equation}
 Define $u_\eps:X\to\RR^m$ by 
$$
u_\eps:=\left\{
\begin{array}{ll}
u&\hbox{in }X\setminus A\\
v_i&\hbox{in }Q_{i}.
\end{array}
\right.
$$
Then $u_\eps-u\in W^{1,p}_{\mu,0}(A;\RR^m)$. Moreover, because of (C$_0$), $\nabla_\mu u_\eps(x)=\nabla_\mu v_i(x)$ for $\mu$-a.a. $x\in Q_i$. From \eqref{EEEqqq1} and \eqref{EEEqqq1-bis} we see that
\begin{equation}\label{EEEqqq0}
\int_A\widehat{L}_x(\nabla_\mu u_\eps(x))d\mu(x)\leq {\rm m}^\eps_u(A)+\eps.
\end{equation}
On the other hand, we have
\begin{eqnarray*}
\|u_\eps-u\|^{p}_{L^{\chi p}_\mu(X;\RR^m)}=\left(\int_A|u_\eps-u|^{\chi p}d\mu\right)^{1\over\chi}&=&\left(\sum_{i\in I}\int_{Q_{i}}|v_i-u|^{\chi p}d\mu\right)^{1\over\chi}\\
&\leq&\sum_{i\in I}\left(\int_{Q_{i}}|v_i-u|^{\chi p}d\mu\right)^{1\over\chi}
\end{eqnarray*}
with $\chi\geq 1$ given by (C$_1$). As $X$ supports a $p$-Sobolev inequality, see (C$_1$) and \eqref{Poincare-Inequality}, and ${\rm diam}(Q_i)\in]0,\eps[$ for all $i\in I$, we have
$$
\|u_\eps-u\|^{p}_{L^{\chi p}_\mu(X;\RR^m)}\leq\eps^{p} {K}^{p}\sum_{i\in I}\int_{Q_{i}}|\nabla_\mu v_i-\nabla_\mu u|^pd\mu
$$
with $K>0$, and so
\begin{equation}\label{EEEqqq2}
\|u_\eps-u\|^{p}_{L^{\chi p}_\mu(X;\RR^m)}\leq 2^p\eps^{p} {K}^{p}\sum_{i\in I}\left(\int_{Q_{i}}|\nabla_\mu v_i|^pd\mu+\int_A|\nabla_\mu u|^pd\mu\right).
\end{equation}
Taking the first inequality in \eqref{EqUaT-ThEoReM-SecoND-ResuLT1} and \eqref{EEEqqq1} into account, from \eqref{EEEqqq2} we deduce that
$$
\|u_\eps-u\|^p_{L^{\chi p}_\mu(X;\RR^m)}\leq 2^p{K}^{p}\eps^{p}\left({1\over C}({\rm m}^\eps_u(A)+\eps)+\int_A|\nabla_\mu u|^pd\mu\right)
$$
which shows that $u_\eps\to u$ in $L^{\chi p}_\mu(X;\RR^m)$ because $\lim_{\eps\to0}{\rm m}_u^\eps(A)={\rm m}^*_u(A)<\infty$. Hence  $u_\eps\to u$ in $L^{p}_\mu(X;\RR^m)$ since $\chi p\geq p$, and \eqref{EEEEqqqq} follows from \eqref{EEEqqq0} by letting $\eps\to 0$ (and by noticing that $\overline{E}(u;A)\leq\liminf_{\eps\to0}\int_A\widehat{L}_x(\nabla_\mu u_\eps(x))d\mu(x)$).

\subsection*{Step 2. We prove that \boldmath$\lim\limits_{\rho\to0}{{\rm m}_u^*(Q_\rho(x))\over\mu(Q_\rho(x))}=\lim\limits_{\rho\to 0}{{\rm m}_u(Q_\rho(x))\over\mu(Q_\rho(x))}$\unboldmath\ for \boldmath$\mu$\unboldmath-a.a. \boldmath$x\in X$\unboldmath}

From Step 1 we have ${\rm m}_u^*=\overline{E}(u;\cdot)$, hence ${\rm m}^*_u\geq{\rm m}_u$ and so $\lim_{\rho\to0}{{\rm m}^*_u(Q_\rho(x))\over\mu(Q_\rho(x))}\geq\limsup_{\rho\to 0}{{\rm m}_u(Q_\rho(x))\over\mu(Q_\rho(x))}$ for $\mu$-a.a. $x\in X$. Thus, it remains to prove that
\begin{equation}\label{EEEEEE_GoAL}
\lim_{\rho\to 0}{{\rm m}^*_u(Q_\rho(x))\over\mu(Q_\rho(x))}\leq\liminf_{\rho\to 0}{{\rm m}_u(Q_\rho(x))\over\mu(Q_\rho(x))}\hbox{ for }\mu\hbox{-a.a. }x\in X.
\end{equation}
Fix any $t>0$. Denote the class of all open balls $Q_\rho(x)$, with $x\in X$ and $\rho>0$, such that ${\rm m}^*_u(Q_{\rho}(x))>{\rm m}_u(Q_{\rho}(x))+t\mu(Q_{\rho}(x))$ by $\mathcal{G}_t$ and define $N_t\subset X$ by
$$
N_t:=\Big\{x\in X:\forall\delta>0\ \exists\rho\in]0,\delta[\ Q_\rho(x)\in\mathcal{G}_t\Big\}.
$$
Fix any $\eps>0$. Using the definition of $N_t$, we can assert that for each $x\in K$ there exists $\{\rho_{x,n}\}_n\subset]0,\eps[$ with $\rho_{x,n}\to0$ as $n\to\infty$ such that for every $n\geq 1$, $\mu(\partial Q_{\rho_{x,n}}(x))=0$ and $Q_{\rho_{x,n}}(x)\in\mathcal{G}_t$. Consider the family $\mathcal{F}_0$ of closed balls in $X$ given by
$$
\mathcal{F}_0:=\left\{\overline{Q}_{\rho_{x,n}}(x):x\in N_t\hbox{ and }n\geq 1\right\}.
$$
Then $\inf\left\{r>0:\overline{Q}_r(x)\in\mathcal{F}_0\right\}=0$ for all $x\in N_t$. As $X$ satisfies the Vitali covering theorem, there exists a disjointed countable subfamily $\{\overline{Q}_i\}_{i\in I_0}$ of closed balls of $\mathcal{F}_0$ (with $\mu(\partial Q_i)=0$ and ${\rm diam}(Q_i)\in]0,\eps[$) such that
$$
N_t\subset\Big(\cupp_{i\in I_0}\overline{Q}_i\Big)\cup\Big(N_t\setminus\cupp_{i\in I_0}\overline{Q}_i\Big)\hbox{ with }\mu\Big(N_t\setminus\cupp_{i\in I_0}\overline{Q}_i\Big)=0. 
$$
If $\mu\big(\cup_{i\in I_0}\overline{Q}_i\big)=0$ then \eqref{EEEEEE_GoAL} will follows. Indeed, in this case we have $\mu(N_t)=0$, i.e., $\mu(X\setminus N_t)=\mu(X)$, and given $x\in X\setminus N_t$ there exists $\delta>0$ such that ${\rm m}^*_u(Q_{\rho}(x))\leq{\rm m}_u(Q_{\rho}(x))+t\mu(Q_{\rho}(x))$ for all $\rho\in]0,\delta[$. Hence $\lim_{\rho\to0}{{\rm m}^*_u(Q_\rho(x))\over\mu(Q_\rho(x))}\leq\liminf_{\rho\to0}{{\rm m}_u(Q_\rho(x))\over\mu(Q_\rho(x))}+t$ for all $t>0$, and \eqref{EEEEEE_GoAL} follows by letting $t\to0$. 

To establish that $\mu\big(\cupp_{i\in I_0}\overline{Q}_i\big)=0$ it is sufficient to prove that for every finite subset $J$ of $I_0$,
\begin{equation}\label{GoAl-PPPRRRoooFFF}
\mu\Big(\cupp_{i\in J}\overline{Q}_i\Big)=0.
\end{equation} 

As $X$ satisfies the Vitali covering theorem and $X\setminus \cupp_{i\in J}\overline{Q}_i$ is open, there exists a countable family $\{B_i\}_{i\in I}$ of disjoint open balls of $X\setminus \cupp_{i\in J}\overline{Q}_i$, with $\mu(\partial B_i)=0$ and ${\rm diam}(B_i)\in]0,\eps[$, such that
\begin{equation}\label{EqUaT-VitaL-1}
\mu\left(\Big(X\setminus\cupp_{i\in J}\overline{Q}_i\Big)\setminus \cupp_{i\in I}B_i\right)=\mu\left(X\setminus\Big(\cupp_{i\in I}B_i\Big)\cup\Big(\cupp_{i\in J}Q_i\Big)\right)=0.
\end{equation}
Recalling that ${\rm m}^*_u$ is the restriction to $\mathcal{O}(X)$ of a finite positive Radon measure  which is absolutely continuous with respect to $\mu$, from \eqref{EqUaT-VitaL-1} we see that
$$
{\rm m}^*_u(X)=\sum_{i\in I}{\rm m}^*_u(B_i)+\sum_{i\in J}{\rm m}^*_u(Q_i).
$$
Moreover, $Q_i\in\mathcal{G}_t$ for all $i\in J$, i.e., ${\rm m}^*_u(Q_i)>{\rm m}_u(Q_i)+t\mu(Q_i)$ for all $i\in J$, and ${\rm m}^*_u\geq {\rm m}_u$, hence 
$$
{\rm m}^*_u(X)\geq \sum_{i\in I}{\rm m}_u(B_i)+\sum_{i\in J}{\rm m}_u(Q_i)+t\mu\left(\cupp_{i\in J}Q_i\right).
$$
As $\{B_i\}_{i\in I}\cup\{Q_i\}_{i\in J}\in\mathcal{V}_\eps(X)$ we have $\sum_{i\in I}{\rm m}_u(B_i)+\sum_{i\in J}{\rm m}_u(Q_i)\geq{\rm m}_u^\eps(X)$, hence ${\rm m}^*_u(X)\geq {\rm m}^\eps_u(X)+t\mu(\cupp_{i\in J}Q_i)$, and \eqref{GoAl-PPPRRRoooFFF} follows by letting $\eps\to0$. \endproof

\subsection{Proof of Theorem \ref{FiNaL-Main-TheOrEM}} Taking Theorem \ref{IR-Theorem} into account it is sufficient to prove that for every $u\in W^{1,p}_\mu(X;\RR^m)$ and $\mu$-a.e. $x\in X$, we have:
\begin{eqnarray}
&&\lim_{\rho\to0}{{\rm m}_u(Q_\rho(x))\over\mu(Q_\rho(x))}\leq\lim_{\rho\to0}{{\rm m}_{u_x}(Q_\rho(x))\over\mu(Q_\rho(x))};\label{FiNaL-EqUa1}\\
&&\lim_{\rho\to0}{{\rm m}_u(Q_\rho(x))\over\mu(Q_\rho(x))}\geq\lim_{\rho\to0}{{\rm m}_{u_x}(Q_\rho(x))\over\mu(Q_\rho(x))},\label{FiNaL-EqUa2}
\end{eqnarray}
where $u_x\in W^{1,p}_\mu(X;\RR^m)$ is given by (A$_1$) (and satisfies \eqref{FinALAssuMpTIOnOne} and \eqref{FinALAssuMpTIOnTwo}) and for each $z\in W^{1,p}_\mu(X;\RR^m)$, ${\rm m}_z:\mathcal{O}(X)\to[0,\infty]$ is defined by
\begin{eqnarray*}
{\rm m}_z(A)&=&\inf\left\{\int_{A}\widehat{L}_y(\nabla_\mu v(y))d\mu(y):v\in W^{1,p}_{\mu,z}(A;\RR^m)\right\}\\
&=&\inf\left\{\int_{A}\widehat{L}_y(\nabla_\mu z(y)+\nabla_\mu w(y))d\mu(y):w\in W^{1,p}_{\mu,0}(A;\RR^m)\right\}.
\end{eqnarray*}
\begin{remark}\label{ReMarK-FiNAl-PrOOf-TheorEM}
From the proof of Theorem \ref{IR-Theorem} we can assert that for every $z\in W^{1,p}_\mu(X;\RR^m)$, the set function ${\rm m}^*_z:\mathcal{O}(X)\to[0,\infty]$ given by 
$$
{\rm m}^*_z(A):=\sup_{\eps>0}\inf\left\{\sum_{i\in I}{\rm m}_z(Q_i):\{Q_i\}_{i\in I}\in \mathcal{V}_\eps(A)\right\}
$$
(where $\mathcal{V}_\eps(A)$ denotes the class of all countable family $\{Q_i\}_{i\in I}$ of disjoint open balls of $A$ with ${\rm diam}(Q_i)\in]0,\eps[$ and $\mu(\partial Q_i)=0$ such that $\mu(A\setminus\cup_{i\in I}Q_i)=0$) is the restriction to $\mathcal{O}(X)$ of a Radon measure on $X$ which absolutely continuous with respect to $\mu$. Moreover, ${\rm m}^*_z\geq{\rm m}_z$ and
\begin{equation}\label{FiNalRemarKEquatION}
\lim_{\rho\to0}{{\rm m}_z^*(Q_\rho(x))\over\mu(Q_\rho(x))}=\lim_{\rho\to 0}{{\rm m}_z(Q_\rho(x))\over\mu(Q_\rho(x))}.
\end{equation}
\end{remark}

We only give the proof of \eqref{FiNaL-EqUa1}. As the proof of \eqref{FiNaL-EqUa2} uses the same method, its detailled verification is left to the reader.

\subsection*{Proof of (\ref{FiNaL-EqUa1})} Fix any $\eps>0$. Fix any $t\in]0,1[$ and any $\rho\in]0,\eps[$. By definition of ${\rm m}_{u_x}(Q_{t\rho}(x))$, where there is no loss of generality in assuming that $\mu(\partial Q_{t\rho}(x))=0$, there exists $w\in W^{1,p}_{\mu,0}(Q_{t\rho}(x);\RR^m)$ such that
\begin{eqnarray}\label{End-PrOOf-EquatION-1}
\int_{Q_{t\rho}(x)}\widehat{L}_y(\nabla_\mu u(x)+\nabla_\mu w(y))d\mu(y)
&\leq& {\rm m}_{u_x}(Q_{t\rho}(x))+\eps\mu(Q_{\rho}(x))\label{FiNalPrOOf-EqN1}\\
&\leq&{\rm m}^*_{u_x}(Q_\rho(x))+\eps\mu(Q_\rho(x)),\nonumber
\end{eqnarray}
where we have used both the fact that ${\rm m}_{u_x}\leq {\rm m}^*_{u_x}$ and ${\rm m}^*_{u_x}$ is increasing (see Remark \ref{ReMarK-FiNAl-PrOOf-TheorEM}). From (A$_2$) there exists a Uryshon function $\varphi\in\mathcal{A}(X)$ for the pair $(X\setminus Q_{\rho}(x),\overline{Q}_{t\rho}(x))$ such that 
\begin{equation}\label{PlAtEAu-FunCtIOn-ProPerTy}
\|D_\mu\varphi\|_{L^p_\mu(X;\RR^N)}\leq {\alpha\over\rho(1-t)}
\end{equation} 
for some $\alpha>0$ (which does not depend on $\rho$). Define $v\in W^{1,p}_{\mu}(Q_\rho(x);\RR^m)$ by
$$
v:=\varphi u_x+(1-\varphi)u.
$$
Then $v-u\in W^{1,p}_{\mu,0}(Q_\rho(x);\RR^m)$. Using \eqref{MU-Gradient-Property} and \eqref{MU-Der-Prop-1} we have
$$
\nabla_\mu v=\left\{
\begin{array}{ll}
\nabla_{\mu}u(x)&\hbox{in }\overline{Q}_{t\rho}(x)\\
D_\mu\varphi\otimes(u_x-u)+\varphi\nabla_\mu u(x)+(1-\varphi)\nabla_\mu u&\hbox{in }Q_\rho(x)\setminus \overline{Q}_{t\rho}(x).
\end{array}
\right.
$$
As $w\in W^{1,p}_{\mu,0}(Q_{t\rho}(x);\RR^m)$ we have $v+w-u\in W^{1,p}_{\mu,0}(Q_\rho(x);\RR^m)$. Noticing that $\mu(\partial Q_{t\rho}(x))=0$ and, because of (C$_0$), $\nabla_\mu w(y)=0$ for $\mu$-a.a. $y\in Q_\rho(x)\setminus \overline{Q}_{t\rho}(x)$ and taking \eqref{End-PrOOf-EquatION-1}, the second inequality in \eqref{EqUaT-ThEoReM-SecoND-ResuLT1} and \eqref{PlAtEAu-FunCtIOn-ProPerTy} into account we deduce that
\begin{eqnarray}
\qquad{{\rm m}_u(Q_\rho(x))\over\mu(Q_\rho(x))}&\leq&\mint_{Q_\rho(x)}\widehat{L}_y(\nabla_\mu v+\nabla_\mu w)d\mu\label{FiRsTEquAtIOnofTHelaSTPrOOf}\\
&=&{1\over\mu(Q_\rho(x))}\int_{\overline{Q}_{t\rho}(x)}\widehat{L}_y(\nabla_\mu u(x)+\nabla_\mu w)d\mu\nonumber\\
&&+{1\over\mu(Q_\rho(x))}\int_{Q_\rho(x)\setminus \overline{Q}_{t\rho}(x)}\widehat{L}_y(\nabla_\mu v)d\mu\nonumber\\
&\leq &{{\rm m}^*_{u_x}(Q_\rho(x))\over\mu(Q_\rho(x))}+\eps\nonumber\\
&& +2^{2p}c\left({\alpha^p\over(1-t)^p}{1\over\rho^p}\mint_{Q_\rho(x)}|u-u_x|^pd\mu+{A_{\rho,t}\over\mu(Q_\rho(x))}\right)\nonumber
\end{eqnarray}
with
$$
A_{\rho,t}:=\mu(Q_\rho(x)\setminus Q_{t\rho}(x))|\nabla_\mu u(x)|^p+\int_{Q_\rho(x)\setminus Q_{t\rho}(x)}|\nabla_\mu u|^pd\mu.
$$
As $\mu$ is a doubling measure, see (A$_3$), we can assert that
$$
\lim_{r\to0}\mint_{Q_r(x)}\big||\nabla_\mu u(y)|^p-|\nabla_\mu u(x)|^p\big|d\mu(y)=0.
$$
But
\begin{eqnarray*}
{A_{\rho,t}\over\mu(Q_\rho(x))}&\leq& 2\left(1-{\mu(Q_{t\rho}(x)\over\mu(Q_\rho(x))}\right)|\nabla_\mu u(x)|^p\\
&&+\mint_{Q_\rho(x)}\big||\nabla_\mu u(y)|^p-|\nabla_\mu u(x)|^p\big|d\mu(y)
\end{eqnarray*}
and so
\begin{equation}\label{LiMItiNrhOofArhot}
\liminf_{\rho\to0}{A_{\rho,t}\over\mu(Q_\rho(x))}\leq 2\left(1-\limsup_{\rho\to0}{\mu(Q_{t\rho}(x))\over\mu(Q_\rho(x))}\right)|\nabla_\mu u(x)|^p.
\end{equation}
Letting $\rho\to 0$ in \eqref{FiRsTEquAtIOnofTHelaSTPrOOf} and using \eqref{FinALAssuMpTIOnTwo} and  \eqref{LiMItiNrhOofArhot} we see that
$$
\lim_{\rho\to0}{{\rm m}_u(Q_\rho(x))\over\mu(Q_\rho(x))}\leq \lim_{\rho\to0}{{\rm m}^*_{u_x}(Q_\rho(x))\over\mu(Q_\rho(x))}+\eps+2\left(1-\limsup_{\rho\to0}{\mu(Q_{t\rho}(x))\over\mu(Q_\rho(x))}\right)|\nabla_\mu u(x)|^p.
$$
Letting $t\to 1$ and using \eqref{DoublINgAssUMpTiON} we conclude that
$$
\lim_{\rho\to0}{{\rm m}_u(Q_\rho(x))\over\mu(Q_\rho(x))}\leq \lim_{\rho\to0}{{\rm m}^*_{u_x}(Q_\rho(x))\over\mu(Q_\rho(x))}+\eps
$$
and (taking \eqref{FiNalRemarKEquatION} into account) \eqref{FiNaL-EqUa1} follows by letting $\eps\to0$. \endproof

\end{document}